\documentclass[a4paper, reqno]{amsart}
\usepackage{amssymb}
\usepackage{fancyhdr}
\usepackage{mathrsfs}
\usepackage{mathtools}
\usepackage{stmaryrd}
\usepackage{amsmath}
\usepackage{tikz-cd}
\usepackage{tikz}
\usepackage{pifont}
\usepackage{color}
\usepackage[all]{xy}
\usepackage{type1cm}
\usepackage{hyperref}
\usepackage{color}
\hypersetup{colorlinks=true,
     breaklinks=true,
     linkcolor=blue,
      pageanchor=true}

\newtheorem{thm}{Theorem}[section]

\newtheorem{cor}[thm]{Corollary}
\newtheorem{lem}[thm]{Lemma}
\newtheorem{prop}[thm]{Proposition}

\theoremstyle{definition}
\newtheorem{defn}[thm]{Definition}
\newtheorem{exam}[thm]{Example}
\newtheorem{rem}[thm]{Remark}

\title[Silting interval reduction and 0-Auslander extriangulated]{Silting interval reduction and 0-Auslander extriangulated categories}
\author{Jixing Pan and Bin Zhu}
\address{J. Pan: Department of  Mathematical Sciences, Tsinghua University, Beijing 100084, P. R. China.}
\email{pjx19@mails.tsinghua.edu.cn}
\address{B. Zhu: Department of  Mathematical Sciences, Tsinghua University, Beijing 100084, P. R. China.}
\email{zhu-b@mail.tsinghua.edu.cn}

\begin{document}

\begin{abstract}
We give a reduction technique for silting intervals in extriangulated categories, which we call "silting interval reduction". It provides a reduction technique for tilting subcategories when the extriangulated categories are exact categories.

In 0-Auslander extriangulated categories (a generalization of the well-known two-term category $K^{[-1,0]}(\mathsf{proj}\Lambda)$ for an Artin algebra $\Lambda$), we provide a reduction theory for silting objects as an application of silting interval reduction. It unifies two-term silting reduction and Iyama-Yoshino's 2-Calabi-Yau reduction. The mutation theory developed by Gorsky, Nakaoka and Palu recently can be deduced from it. Since there are bijections between the silting objects and the support $\tau$-tilting modules over certain finite dimensional algebras, we show it is compatible with $\tau$-tilting reduction. This compatibility theorem also unifies the two compatibility theorems obtained by Jasso in his work on $\tau$-tilting reduction.

We give a new construction for 0-Auslander extriangulated categories using silting mutation, together with silting interval reduction, we obtain some results on silting quivers. Finally, we prove that $d$-Auslander extriangulated categories are related to a certain sequence of silting mutations.
\end{abstract}

\maketitle

\tableofcontents

\section{Introduction}

As a generalization of tilting theory, classical silting theory in triangulated categories was introduced by Keller and Vossieck \cite{KV}, and rediscovered by Aihara and Iyama \cite{AI}. There are two main techniques in silting theory called mutation and reduction (e.g. \cite{AI}, \cite{IY18}). Mutation allows us to construct a new silting object (resp. subcategory) from a given one by replacing a direct summand (resp. some objects). Reduction concerns a collection of silting objects/subcategories in a category with a common part, then bijectively reduces them to a smaller category.  Recently, silting theory was generalized to extriangulated categories (e.g. \cite{AT22}, \cite{AT23}, \cite{LZZZ}).

Extriangulated categories were introduced by Nakaoka and Palu \cite{NP}, which are a simultaneous generalization of exact categories and triangulated categories. There are more examples such as extension closed subcategories, ideal quotient by projective-injectives \cite{NP}, relative substructure (e.g. \cite{ZZ}, \cite{ZH}, \cite{HLN}) and localization (e.g. \cite{NP}, \cite{NOS}). This greater framework allows us to study silting theory from a unified point of view.  Gorsky, Nakaoka and Palu \cite{GNP23} introduced the notion of 0-Auslander extriangulated categories. It is a generalization of the homotopy categories of two-term complexes of projectives, and more generally, two-term categories $\mathcal{R}\ast \mathcal{R}[1]$ for a rigid subcategory $\mathcal{R}$ in a triangulated category (e.g. \cite{YZ}, \cite{YZZ19}, \cite{FGL}, \cite{ZZ20}, \cite{Dong Yang}) and so on (see \cite[Section 3]{GNP23}). It allows a well-performed mutation theory (see \cite[Theorem 1.3]{GNP23}), which generalizes the well-known mutation theory of two-term silting objects, and relative cluster tilting objects \cite{YZ}.

As we know, reduction usually concerns a collection of objects/subcategories with a common part, then bijectively reduces them to a subcategory (or its quotient). In $\tau$-tilting reduction \cite{G.Jasso}, this collection becomes an interval in the poset of all basic support $\tau$-tiling pairs in $\mathsf{mod}A$. There are more results concerning bijections of intervals such as \cite{AET} and \cite{AT22}. Thus these results can also be regarded as reduction techniques. Motivated by such results, we provide a reduction technique for intervals of silting subcategories in extriangulated categories and call it "silting interval reduction". It allows us to study silting theory from a new perspective. Note that though we use the word "reduction", it is different from classical silting reduction (\cite{AI}, \cite{IY18}, \cite{LZZZ}). Recently, Børve \cite{Borve} generalized Iyama-Yang's classical silting reduction to extriangulated categories. Then he applied his result to 0-Auslander categories and obtained one of our main results (Theorem \ref{1.2}) again.

This paper is organized as follows. In Section \ref{section 2}, we recall some basic notions, notations and results which are needed in the following sections. In Section \ref{section 3}, we study intervals of cotorsion pairs. Then we obtain our first main theorem and call it silting interval reduction. Denote by $\mathsf{silt}\,\mathcal{B}$ the poset of silting subcategories in an extriangulated category $\mathcal{B}$. Since it is a poset, we can define intervals (see Section \ref{section 2}).

\begin{thm}{\rm (Theorem \ref{main thm_1})}
	Let $\mathcal{C}$ be an extriangulated category and $\mathcal{M},\mathcal{N}$ be silting subcategories such that $\mathcal{M}\leq \mathcal{N}$. Then $\mathsf{silt}(-\infty,\mathcal{N}]\cong \mathsf{silt}(\mathcal{N}^{\bot})$, $\mathsf{silt}[\mathcal{M},+\infty)\cong \mathsf{silt}({^{\bot}\mathcal{M}})$ and $\mathsf{silt}\,[\mathcal{M},\mathcal{N}]\cong \mathsf{silt}({^{\bot}\mathcal{M}}\cap \mathcal{N}^{\bot})$ as posets.
\end{thm}

As immediate applications, we apply it to triangulated categories and exact categories. In triangulated categories, we generalize \cite[Theorem 2.3]{IJY}, \cite[Theorem 2.2]{PZ}, \cite[Corollary 1.22]{BZ} and \cite[Corollary 5.10]{GNP23} to $n$-term categories (see Corollary \ref{generalize PZ}). Moreover we prove that $n$-term silting subcategories are precisely silting subcategories in the $n$-term extriangulated categories, which generalizes \cite[Theorem 3.4(2)]{FGL}. In exact categories, we obtain a reduction technique for intervals of tilting subcategories in the sense of Sauter \cite{Sauter} (see Proposition \ref{appli in exact cat}).

In Section \ref{section 4}, we focus on 0-Auslander extriangulated categories (Definition \ref{defn of 0-Aus}). Let $\mathcal{C}$ be a $k$-linear ($k$ is a field), Hom-finite and Krull-Schmidt reduced 0-Auslander extriangulated category with $\mathsf{proj}\,\mathcal{C}=\mathsf{add}P$ and $A={\rm End}_{\mathcal{C}}(P)$. First we prove there are bijections between silting objects (resp. cotorsion pairs) in $\mathcal{C}$ and support $\tau$-tilting pairs (resp. left weak cotorsion-torsion triples) in $\mathsf{mod}A$ (see Corollary \ref{silt-tau-tilt} and Proposition \ref{cotors-lwcotorstors}). It generalizes many well-known results (e.g. \cite{AIR}, \cite{IJY}, \cite{YZ}, \cite{FGL}, \cite{PZ}, \cite{BZ}).

Next we give a reduction theory in 0-Auslander categories as an application of silting interval reduction. Let $U$ be a basic presilting (not silting) object and $N,M$ be the Bongartz and co-Bongartz completions of $U$ respectively (Lemma \ref{two completions}). Define $\mathcal{C}':={^{\bot}M}\cap N^{\bot}$ and $\widetilde{(-)}:\mathcal{C}'\rightarrow \widetilde{\mathcal{C}'}:=\mathcal{C}'/[\mathsf{add}U]$. Denote by $\mathsf{silt}_U\mathcal{C}$ the set of basic silting objects with direct summand $U$. Then we have 

\begin{thm}{\rm (Theorem \ref{reduction in 0-Aus})}\label{1.2}
	$\widetilde{(-)}$ induces an isomorphism of posets
	\[{\rm red}:\mathsf{silt}_{U}\mathcal{C}\rightarrow \mathsf{silt}\,\widetilde{\mathcal{C}'}.\]
\end{thm}

As an application, we recover the mutation theory in $\mathcal{C}$ (see \cite[Theorem 1.3]{GNP23}). That is, "reduction" implies "mutation", which is similar to \cite[Corollary 3.18]{G.Jasso}.

\begin{prop}{\em (Theorem \ref{another proof for GNP1.3})}
	Let $U\in \mathcal{C}$ be a basic presilting (not silting) object. There exists objects $X'$ and $Y'$ which are not isomorphic, such that $N:=X'\oplus U$ and $M:=Y'\oplus U$ are basic silting objects and $M\leq N$. Moreover $\mathsf{silt}\,[M,N]=\mathsf{silt}_{U}\mathcal{C}$ and there is an $\mathbb{E}$-triangle
	\[X'\stackrel{f}\rightarrowtail U_{1}\stackrel{g}\twoheadrightarrow Y'\dashrightarrow \]
	with $f$ (resp. $g$) a minimal left (resp. right) $\mathsf{add}U$-approximation. In particular, if $U$ is almost complete (i.e. $|U|=|P|-1$), then $\mathsf{silt}_{U}\mathcal{C}=\{M,N\}$ and this is just \cite[Theorem 1.3]{GNP23}.
\end{prop}

Next we show the reduction map ${\rm red}:\mathsf{silt}_{U}\mathcal{C}\rightarrow \mathsf{silt}\,\widetilde{\mathcal{C}'}$ is compatible with Jasso's $\tau$-tilting reduction as follows. This result generalizes \cite[Theorem 4.12]{G.Jasso} and the reduction map ${\rm red}:\mathsf{silt}_U\mathcal{C}\rightarrow \mathsf{silt}\,\widetilde{\mathcal{C}'}$ generalizes two-term silting reduction map (\cite[Corollary 4.16]{G.Jasso}).

\begin{thm}{\rm (Theorem \ref{main thm2})}\label{thm3 in intro}
	There is a commutative diagram
	\[\begin{tikzcd}
		\mathsf{silt}_{U}\mathcal{C} \arrow[rr,"\overline{(-)}={\rm Hom}_{\mathcal{C}}(P{,}-)"] \arrow[d,"{\rm red}"] & {} & s\tau\text{-}\mathsf{tilt}_{(\overline{U},\overline{\Omega I_{U}})}A \arrow[d,"{\rm red}"] \\
		\mathsf{silt}\,\widetilde{\mathcal{C}'} \arrow[rr,"{\rm Hom}_{\widetilde{\mathcal{C}'}}(N{,}-)"] & {} & s\tau\text{-}\mathsf{tilt}B
	\end{tikzcd}\]
	in which all maps are isomorphisms of posets.
\end{thm}

In Section \ref{section 5}, we provide a new construction of 0-Auslander extriangulated categories using silting mutations.

\begin{thm}{\rm (Theorem \ref{main thm: new 0-Aus})}\label{thm6 in intro}
	Let $\mathcal{C}$ be an extriangulated category, $\mathcal{M},\mathcal{N}\in \mathsf{silt}\,\mathcal{C}$ and $\mathcal{C}':={^{\bot}\mathcal{M}}\cap \mathcal{N}^{\bot}$. Then $\mathcal{M}\leq \mathcal{N}$ and $\mathcal{C}'$ is 0-Auslander if and only if $\mathcal{M}$ is a left mutation of $\mathcal{N}$ with respect to a good covariantly finite subcategory $\mathcal{D}$ of $\mathcal{N}$.
\end{thm}

It can be interpreted as "0-Auslander" is equivalent to "one step of silting mutation". We show that some well-known examples of 0-Auslander extriangulated categories can also arise from silting mutation (Example \ref{examples}). Using silting interval reduction and the above theorem, we obtain several direct consequences on silting mutation. For example, under certain conditions, silting subcategories between a mutation pair form an $n$-regular subquiver of the silting quiver for some $n$ (Corollary \ref{silting between mutation}). We also show that a left mutation can be obtained by a sequence of irreducible left mutations if and only if a certain algebra admits a maximal green sequence (Corollary \ref{mutation vs maximal green sequence}).

2-Calabi-Yau triangulated category with a cluster tilting object can be regarded as a reduced 0-Auslander extriangulated category, and cluster tilting objects become silting objects (see \cite{YZ}, \cite{FGL}, \cite{PPPP23}, \cite{GNP23}). Thus reduction theory in 0-Auslander categories can also be applied to this situation. Indeed, we show in Section \ref{section 6} that it recovers a general version of 2-Calabi-Yau reduction introduced by Faber, Marsh and Pressland \cite{FMP} recently. Then we see Theorem \ref{thm3 in intro} also covers \cite[Theorem 4.24]{G.Jasso}.

\begin{thm}{\rm (Proposition \ref{relative struc on stably 2-CY}, Theorem \ref{our reduction vs 2-CY reduction})}
	Let $(\mathcal{C},\mathbb{E},\mathfrak{s})$ be a stably 2-Calabi-Yau Frobenius extriangulated category with a contravariantly finite subcategory $\mathcal{T}$ of $\mathcal{C}$ such that $\mathcal{T}=\mathcal{T}^{\bot_{1}}$. Let $\mathcal{X}$ be a functorially finite rigid subcategory of $\mathcal{C}$. Consider the relative substructure $\mathbb{E}_{\mathcal{T}}$ on $\mathcal{C}$ determined by $\mathcal{T}$, then $(\mathcal{C},\mathbb{E}_{\mathcal{T}})$ is 0-Auslander and cluster tilting subcategories in $(\mathcal{C},\mathbb{E})$ coincide with functorially finite silting subcategories in $(\mathcal{C},\mathbb{E}_{\mathcal{T}})$. Moreover, reduction theory in 0-Auslander categories recovers the reduction in \cite[Theorem 3.21]{FMP}.
\end{thm}

In Section \ref{section 7}, we generalize Theorem \ref{thm6 in intro} as follows.

\begin{thm}{\rm (Theorem \ref{d-Auslander vs mutation})}
Let $\mathcal{M},\mathcal{N}\in \mathsf{silt}\,\mathcal{C}$, $\mathcal{C}':={^{\bot}\mathcal{M}}\cap \mathcal{N}^{\bot}$ and $d\geq 0$ an integer. Then $\mathcal{M}\leq \mathcal{N}$ and $\mathcal{C}'$ is $d$-Auslander if and only if there is a sequence of silting subcategories 
\[\mathcal{N}=\mathcal{T}_{0}\geq \mathcal{T}_{1}\geq \cdots \geq \mathcal{T}_{d+1}=\mathcal{M}\]
such that for $0\leq i\leq d$, $\mathcal{T}_{i+1}=\mu^{L}(\mathcal{T}_{i};\mathcal{D})$ for a good covariantly finite subcategory $\mathcal{D}$ of $\mathcal{T}_{i}$.
\end{thm}

\section{Preliminaries}\label{section 2}

Let $\mathcal{C}=(\mathcal{C},\mathbb{E},\mathfrak{s})$ be a $k$-linear {\em extriangulated category} ($k$ is a commutative ring with 1) where $\mathcal{C}$ is an $k$-linear additive category, $\mathbb{E}:\mathcal{C}^{op}\times \mathcal{C}\rightarrow {\rm Mod}k$ is a $k$-bilinear functor (sometimes we write $\mathbb{E}_{\mathcal{C}}$ to emphasize the category) and $\mathfrak{s}$ is a {\em realization} of $\mathbb{E}$. We omit the definition but refer to \cite{NP} \cite{Palu23} for details. For an $\mathbb{E}$-{\em extension} $\delta \in \mathbb{E}(C,A)$, if $\mathfrak{s}(\delta)=[A\stackrel{a}\rightarrow B\stackrel{b}\rightarrow C]$, then we call the complex $A\stackrel{a}\rightarrow B\stackrel{b}\rightarrow C$ an ($\mathfrak{s}$-){\em conflation} and $a$ (resp. $b$) an ($\mathfrak{s}$-){\em inflation} (resp. ($\mathfrak{s}$-){\em deflation}). We write the pair ($A\stackrel{a}\rightarrow B\stackrel{b}\rightarrow C$, $\delta$) as $A\stackrel{a}\rightarrowtail B\stackrel{b}\twoheadrightarrow C\stackrel{\delta}\dashrightarrow$ and call it an $\mathbb{E}$-{\em triangle}. Throughout this paper $\mathcal{C}$ is an essentially small $k$-linear extriangulated category. All subcategories are full and closed under isomorphisms.

Let $X\in \mathcal{C}$ and $\mathcal{D}\subseteq \mathcal{C}$ be a subcategory. A morphism $f:X\rightarrow D$ with $D\in \mathcal{D}$ is a {\em left} $\mathcal{D}$-{\em approximation} of $X$ if $f^{\ast}:{\rm Hom}_{\mathcal{C}}(D,D')\rightarrow {\rm Hom}_{\mathcal{C}}(X,D')$ is surjective for each $D'\in \mathcal{D}$. Let $\mathcal{M}$ be a subcategory of $\mathcal{C}$ such that $\mathcal{D}\subseteq \mathcal{M}$. $\mathcal{D}$ is {\em covariantly finite} in $\mathcal{M}$ if each $M\in \mathcal{M}$ has a left $\mathcal{D}$-approximation. Dually, we define a morphism $g:D\rightarrow X$ to be a {\em right} $\mathcal{D}$-{\em approximation} of $X$ and $\mathcal{D}$ to be {\em contravariantly finite} in $\mathcal{M}$. $\mathcal{D}$ is {\em functorially finite} in $\mathcal{M}$ if it is both covariantly finite and contravariantly finite in $\mathcal{M}$. Let $\mathsf{add}\mathcal{D}$ denote the smallest subcategory of $\mathcal{C}$ containing $\mathcal{D}$ and closed under finite direct sums and direct summands.

For an additive subcategory $\mathcal{D}$ of $\mathcal{C}$, denote by $[\mathcal{D}]$ the ideal of $\mathcal{C}$ consisting of morphisms that factor through an object in $\mathcal{D}$. That is, for any $X,Y\in \mathcal{C}$, we have $[\mathcal{D}](X,Y)=\{f\in {\rm Hom}_{\mathcal{C}}(X,Y)\,|\,f\text{ factors through an object in }\mathcal{D}\}$.

For any two subcategories $\mathcal{X}$, $\mathcal{Y}$ in $\mathcal{C}$, define three subcategories as follows:
\[\mathcal{X}\ast \mathcal{Y}:=\{M\in \mathcal{C}\,|\,\exists~\mathbb{E}\text{-triangle}~X\rightarrowtail M \twoheadrightarrow Y\dashrightarrow~\text{with}~X\in \mathcal{X}~\text{and}~Y\in \mathcal{Y}\},\]
\[{\rm Cone(\mathcal{X},\mathcal{Y})}:=\{M\in \mathcal{C}\,|\,\exists~\mathbb{E}\text{-triangle}~X\rightarrowtail Y \twoheadrightarrow M\dashrightarrow~\text{with}~X\in \mathcal{X}~\text{and}~Y\in \mathcal{Y}\},\]
\[{\rm Cocone(\mathcal{X},\mathcal{Y})}:=\{M\in \mathcal{C}\,|\,\exists~\mathbb{E}\text{-triangle}~M\rightarrowtail X \twoheadrightarrow Y\dashrightarrow~\text{with}~X\in \mathcal{X}~\text{and}~Y\in \mathcal{Y}\}.\]
Note that the definitions above rely on the category $\mathcal{C}$. Thus to be precise, we should write them as $\mathcal{X}\ast_{\mathcal{C}} \mathcal{Y},{\rm Cone_{\mathcal{C}}(\mathcal{X},\mathcal{Y})}$ and ${\rm Cocone_{\mathcal{C}}(\mathcal{X},\mathcal{Y})}$. But if there is no confusion, we often omit the subscript. We say $\mathcal{X}$ is {\em closed under cones} (resp. {\em closed under cocones, closed under extensions}) if Cone$(\mathcal{X},\mathcal{X})\subseteq \mathcal{X}$ (resp. Cocone$(\mathcal{X},\mathcal{X})\subseteq \mathcal{X}$, $\mathcal{X}\ast \mathcal{X}\subseteq \mathcal{X}$). A {\em thick} subcategory is a subcategory which is closed under cones, cocones, extensions and summands. Denote by $\mathsf{thick}\mathcal{X}$ the smallest thick subcategory containing $\mathcal{X}$. Note that for an extension closed subcategory $\mathcal{D}$ of $\mathcal{C}$, there is a restricted extriangulated structure on $\mathcal{D}$. That is, $(\mathcal{D},\mathbb{E}_{\mathcal{D}},\mathfrak{s}_{\mathcal{D}})$ is extriangulated (see \cite[Remark 2.18]{NP}).

For a subcategory $\mathcal{X}$ of $\mathcal{C}$, let $\mathcal{M}^{\wedge}:=\displaystyle\bigcup_{n\geq 0}\mathcal{M}_{n}^{\wedge}$ (resp. $\mathcal{M}^{\vee}:=\displaystyle\bigcup_{n\geq 0}\mathcal{M}_{n}^{\vee}$) where $\mathcal{M}_{n}^{\wedge}$ (resp. $\mathcal{M}_{n}^{\vee}$) is defined inductively as
\[\mathcal{M}_{n}^{\wedge}:={\rm Cone}(\mathcal{M}_{n-1}^{\wedge},\mathcal{M}), \,\mathcal{M}_{0}^{\wedge}=\mathcal{M}\,({\rm resp.} \mathcal{M}_{n}^{\vee}:={\rm Cocone}(\mathcal{M},\mathcal{M}_{n-1}^{\vee}), \, \mathcal{M}_{0}^{\vee}=\mathcal{M}).\]

In \cite{GNP21}, the authors defined higher extensions $\mathbb{E}^{n}$ for $n\geq 1$. They are $k$-bilinear functors $\mathbb{E}^{n}:\mathcal{C}^{\rm op}\times \mathcal{C}\rightarrow {\rm Mod}k$. Note that $\mathbb{E}^{1}=\mathbb{E}$ and for an extension closed subcategory $\mathcal{D}$ of $\mathcal{C}$, we have $\mathbb{E}_{\mathcal{D}}(X,Y)=\mathbb{E}(X,Y)$ but not necessarily $\mathbb{E}_{\mathcal{D}}^{n}(X,Y)\cong \mathbb{E}^{n}(X,Y)$ (see \cite[Remark 3.29]{GNP21}). For any $\mathbb{E}$-triangle $A\stackrel{a}\rightarrowtail B\stackrel{b}\twoheadrightarrow C\stackrel{\delta}\dashrightarrow$ and any $X\in \mathcal{C}$, there are two long exact sequences (see \cite[Theorem 3.5]{GNP21})
\begin{equation*}
	\begin{split}
		{\rm Hom}_{\mathcal{C}}(X,A) & \rightarrow {\rm Hom}_{\mathcal{C}}(X,B)\rightarrow {\rm Hom}_{\mathcal{C}}(X,C)\rightarrow \mathbb{E}(X,A)\rightarrow \mathbb{E}(X,B) \\
		& \rightarrow \mathbb{E}(X,C)\rightarrow \mathbb{E}^{2}(X,A)\rightarrow \cdots
	\end{split}
\end{equation*}
and
\begin{equation*}
	\begin{split}
		{\rm Hom}_{\mathcal{C}}(C,X) & \rightarrow {\rm Hom}_{\mathcal{C}}(B,X)\rightarrow {\rm Hom}_{\mathcal{C}}(A,X)\rightarrow \mathbb{E}(C,X)\rightarrow \mathbb{E}(B,X) \\ 
		& \rightarrow \mathbb{E}(A,X)\rightarrow \mathbb{E}^{2}(C,X)\rightarrow \cdots .
	\end{split}
\end{equation*}
If $\mathcal{C}$ has enough projectives and injectives, then the higher extensions are isomorphic to those defined in \cite{LN} and \cite{HLN} (see \cite[Corollary 3.21]{GNP21}). For a subcategory $\mathcal{X}$ in $\mathcal{C}$, we define
\[\mathcal{X}^{\bot_{1}}:=\{M\in \mathcal{C}\,|\,\mathbb{E}(X,M)=0,~\forall X\in \mathcal{X}\}\]
and
\[\mathcal{X}^{\bot}:=\{M\in \mathcal{C}\,|\,\mathbb{E}^{\geq 1}(X,M)=0,~\forall X\in \mathcal{X}\}.\]
Dually we define ${^{\bot_{1}}}\mathcal{X}$ and ${^{\bot}}\mathcal{X}$. For an object $X\in \mathcal{C}$, define
\[{\rm pd}_{\mathcal{C}}X:={\rm sup}\{n\,|\,\mathbb{E}^{n}(X,Y)\neq 0\text{ for some object }Y\in \mathcal{C}\}.\]
Dually we define ${\rm id}_{\mathcal{C}}X$.

\begin{defn}
	Let $(\mathcal{X},\mathcal{Y})$ be a pair of subcategories in $\mathcal{C}$. It is a {\em (complete) cotorsion pair} if it satisfies
	
	(1) $\mathcal{X}$ and $\mathcal{Y}$ are closed under taking direct summands,
	
	(2) $\mathbb{E}(\mathcal{X},\mathcal{Y})=0$,
	
	(3) $\mathcal{C}$ = Cone($\mathcal{Y}$, $\mathcal{X}$) = Cocone($\mathcal{Y}$, $\mathcal{X}$).

	If moreover $\mathbb{E}^{2}(\mathcal{X},\mathcal{Y})=0$ (resp. $\mathbb{E}^{\geq 2}(\mathcal{X},\mathcal{Y})=0$), it is called an {\em s-cotorsion pair} (resp. {\em hereditary cotorsion pair}). A cotorsion pair $(\mathcal{X},\mathcal{Y})$ is {\em bounded} if $\mathcal{C}=\mathcal{X}^{\wedge}=\mathcal{Y}^{\vee}$ (see \cite{AT22}).
\end{defn}

We denote by $\mathsf{cotors}\,\mathcal{C}$ the set of cotorsion pairs in $\mathcal{C}$, and by $\mathsf{scotors}\,\mathcal{C}$ (resp. $\mathsf{hcotors}\,\mathcal{C}$) the set of {\em s}-cotorsion pairs (resp. hereditary cotorsion pairs) in $\mathcal{C}$. We add a prefix "$\mathsf{bdd}\text{-}$" for bounded cotorsion pairs. 

For $x_{i}=(\mathcal{X}_{i},\mathcal{Y}_{i})\in \mathsf{cotors}\,\mathcal{C},~i=1,2$, write $(\mathcal{X}_{1},\mathcal{Y}_{1})\leq (\mathcal{X}_{2},\mathcal{Y}_{2})$ if $\mathcal{Y}_{1}\subseteq \mathcal{Y}_{2}$. Then $(\mathsf{cotors}\,\mathcal{C},\leq)$ becomes a poset. Let $x_{1},x_{2}\in \mathsf{cotors}\,\mathcal{C}$ such that $x_{1}\leq x_{2}$. We define an interval
\[\mathsf{cotors}\,[x_{1},x_{2}]:=\{x\in \mathsf{cotors}\,\mathcal{C}\,|\,x_{1}\leq x\leq x_{2}\}.\]
Similarly we define
\[\mathsf{cotors}\,(-\infty, x_{2}]:=\{x\in \mathsf{cotors}\,\mathcal{C}\,|\,x\leq x_{2}\}\]
and
\[\mathsf{cotors}\,[x_{1}, +\infty):=\{x\in \mathsf{cotors}\,\mathcal{C}\,|\,x_{1}\leq x \}.\]

\begin{lem}\label{scotorsion}
	Let $(\mathcal{X},\mathcal{Y})$ be a cotorsion pair in $\mathcal{C}$.
	
	{\rm (1)} \cite[Lemma 3.2]{AT22} If it is an s-cotorsion pair, then $\mathcal{X}$ is closed under cocones and $\mathcal{Y}$ is closed under cones.
	
	{\rm (2)} If $\mathcal{C}$ has enough projectives or injectives, then s-cotorsion pairs are hereditary cotorsion pairs.
\end{lem}
\begin{proof}
	(2) is also mentioned in \cite{AT22}. It follows from (1) and \cite[Lemma 4.3]{LZ20}.
\end{proof}

\begin{defn}
	An additive subcategory $\mathcal{X}$ of $\mathcal{C}$ is called {\em presilting} if it is self-orthogonal (i.e. $\mathbb{E}^{\geq 1}(\mathcal{X},\mathcal{X})=0$) and closed under direct summands. $\mathcal{X}$ is called {\em silting} if it is presilting and  $\mathsf{thick}\mathcal{X}=\mathcal{C}$. Denote by $\mathsf{presilt}\,\mathcal{C}$ (resp. $\mathsf{silt}\,\mathcal{C}$) the set of all presilting (resp. silting) subcategories in $\mathcal{C}$.
\end{defn}

\begin{rem}
	For $\mathcal{M},\mathcal{N}\in \mathsf{silt}\,\mathcal{C}$, write $\mathcal{M}\leq \mathcal{N}$ if $\mathbb{E}^{\geq 1}(\mathcal{N},\mathcal{M})=0$. Then $(\mathsf{silt}\,\mathcal{C},\leq)$ is a poset (see \cite[Proposition 5.12]{AT22}). Similarly, for $\mathcal{M},\mathcal{N}\in \mathsf{silt}\,\mathcal{C}$ such that $\mathcal{M}\leq \mathcal{N}$, we define  $\mathsf{silt}\,[\mathcal{M},\mathcal{N}]$, $\mathsf{silt}\,(-\infty,\mathcal{N}]$ and $\mathsf{silt}\,[\mathcal{M}, +\infty)$. 
\end{rem}

The relation between cotorsion pairs and silting subcategories is as follows.

\begin{thm}\cite[Theorem 5.7]{AT22}\label{AT22thm5.7}
	There are mutually inverse isomorphisms of posets
	\begin{align*}
		\mathsf{bdd\text{-}hcotors}\,\mathcal{C} & \leftrightarrow \mathsf{silt}\,\mathcal{C} \\
		(\mathcal{X},\mathcal{Y}) & \mapsto \mathcal{X}\cap \mathcal{Y} \\
		({^{\bot}\mathcal{M}},\mathcal{M}^{\bot})=(\mathcal{M}^{\vee},\mathcal{M}^{\wedge}) & \mapsfrom \mathcal{M}
	\end{align*}
\end{thm}

\section{Silting interval reduction}\label{section 3}

Let $\mathcal{C}=(\mathcal{C},\mathbb{E},\mathfrak{s})$ be an extriangulated category. 

\begin{lem}\label{bddscotors}
	Let $(\mathcal{X},\mathcal{Y})\in \mathsf{bdd\text{-}scotors}\,\mathcal{C}$ and $\mathcal{M}=\mathcal{X}\cap \mathcal{Y}$. Then $\mathcal{X}=\mathcal{M}^{\vee}$ and $\mathcal{Y}=\mathcal{M}^{\wedge}$.
\end{lem}
\begin{proof}
	It follows from the proof of \cite[Proposition 5.9]{AT22}.
\end{proof}

\begin{prop}\label{first prop}
	Let $x=(\mathcal{X},\mathcal{Y})\in \mathsf{cotors}\,\mathcal{C}$. Then $\mathcal{Y}$ is an extriangulated category with enough projectives $\mathcal{N}:=\mathcal{X}\cap \mathcal{Y}$.
	
	(1) If $x\in \mathsf{scotors}\,\mathcal{C}$. Then $\mathbb{E}_{\mathcal{Y}}^{2}(T,S)\cong \mathbb{E}^{2}(T,S)$ for any $T,S\in \mathcal{Y}$ and there are mutually inverse isomorphisms of posets
	\begin{align*}
		\mathsf{cotors}(-\infty,x] & \leftrightarrow \mathsf{cotors}\,\mathcal{Y} \\
		(\mathcal{U},\mathcal{V}) & \mapsto (\mathcal{U}\cap \mathcal{Y},\mathcal{V}) \\
		(\mathsf{add}(\mathcal{X}\ast \mathcal{A}),\mathcal{B}) & \mapsfrom (\mathcal{A},\mathcal{B}).
	\end{align*}
    They restrict to isomorphisms $\mathsf{scotors}(-\infty,x] \leftrightarrow \mathsf{scotors}\,\mathcal{Y}=\mathsf{hcotors}\,\mathcal{Y}$.
    
    (2) If $x\in \mathsf{hcotors}\,\mathcal{C}$, then $\mathbb{E}_{\mathcal{Y}}^{m}(T,S)\cong \mathbb{E}^{m}(T,S)$ for any $T,S\in \mathcal{Y}$ and $m\geq 1$, $\mathsf{scotors}(-\infty,x]=\mathsf{hcotors}(-\infty,x]$ and $\mathsf{presilt}\,\mathcal{Y}=\{\mathcal{T}\in \mathsf{presilt}\,\mathcal{C}\,|\,\mathcal{T}\subseteq  \mathcal{Y}\}$.
    
    (3) If $x\in \mathsf{bdd\text{-}scotors}\,\mathcal{C}$. Then each $T\in \mathcal{Y}$ satisfies ${\rm pd}_{\mathcal{Y}}T<\infty$. If $(\mathcal{U},\mathcal{V})\leq (\mathcal{X},\mathcal{Y})$ is bounded in $\mathcal{C}$, then so is $(\mathcal{U}\cap \mathcal{Y},\mathcal{V})$ in $\mathcal{Y}$.
    
    (4) If $x\in \mathsf{bdd\text{-}hcotors}\,\mathcal{C}$, then isomorphisms in (1) restrict to isomorphisms 
    \[\mathsf{bdd\text{-}hcotors}(-\infty,x] \leftrightarrow \mathsf{bdd\text{-}hcotors}\,\mathcal{Y}.\]
    
\end{prop}
\begin{proof}
	Since $(\mathcal{X},\mathcal{Y})$ is a cotorsion pair, $\mathcal{Y}$ is an extriangulated category. For $T\in \mathcal{Y}$, there is an $\mathbb{E}$-triangle $Y\rightarrowtail X\twoheadrightarrow T\dashrightarrow $ with $X\in \mathcal{X}$ and $Y\in \mathcal{Y}$. Thus $X\in \mathcal{X}\cap \mathcal{Y}=\mathcal{N}$ and it is an $\mathbb{E}_{\mathcal{Y}}$-triangle. Since $\mathbb{E}_{\mathcal{Y}}(\mathcal{N},\mathcal{Y})=\mathbb{E}(\mathcal{N},\mathcal{Y})=0$, $\mathcal{Y}$ has enough projectives $\mathcal{N}$. 
	
	(1) For $S\in \mathcal{Y}$, apply ${\rm Hom}_{\mathcal{C}}(-,S)$ to the $\mathbb{E}$-triangle above. We have $\mathbb{E}_{\mathcal{Y}}^{2}(T,S)\cong \mathbb{E}_{\mathcal{Y}}(Y,S)=\mathbb{E}(Y,S)\cong \mathbb{E}^{2}(T,S)$ since $\mathbb{E}^{2}(\mathcal{X},\mathcal{Y})=0$. Let $(\mathcal{U},\mathcal{V})$ be a cotorsion pair such that $\mathcal{V}\subseteq \mathcal{Y}$. Then $\mathbb{E}_{\mathcal{Y}}(\mathcal{U}\cap \mathcal{Y},\mathcal{V})=\mathbb{E}(\mathcal{U}\cap \mathcal{Y},\mathcal{V})=0$. For any $T\in \mathcal{Y}$, there are two $\mathbb{E}$-triangles $T\rightarrowtail V_{1}\twoheadrightarrow U_{1}\dashrightarrow $ and $V_{2}\rightarrowtail U_{2}\twoheadrightarrow T\dashrightarrow $ with $U_{1},U_{2}\in \mathcal{U}$ and $V_{1},V_{2}\in \mathcal{V}$. Since $\mathcal{Y}$ is closed under extensions and cones by Lemma \ref{scotorsion}, $U_{2}\in \mathcal{U}\cap \mathcal{Y}$ and $U_{1}\in \mathcal{U}\cap \mathcal{Y}$. They are also  $\mathbb{E}_{\mathcal{Y}}$-triangles.
	
	Let $(\mathcal{A},\mathcal{B})\in \mathsf{cotors}\,\mathcal{Y}$. Since $\mathbb{E}(\mathcal{A},\mathcal{B})=\mathbb{E}_{\mathcal{Y}}(\mathcal{A},\mathcal{B})=0$, together with $\mathbb{E}(\mathcal{X},\mathcal{B})=0$, we obtain $\mathbb{E}(\mathcal{X}\ast \mathcal{A},\mathcal{B})=0$. We have $\mathcal{C}={\rm Cone}(\mathcal{Y},\mathcal{X})\subseteq {\rm Cone}({\rm Cocone}(\mathcal{B},\mathcal{A}),\mathcal{X})\subseteq {\rm Cone}(\mathcal{B},\mathcal{X}\ast \mathcal{A})$ and $\mathcal{C}={\rm Cocone}(\mathcal{Y},\mathcal{X})\subseteq {\rm Cocone}({\rm Cocone}(\mathcal{B},\mathcal{A}),\mathcal{X})\subseteq {\rm Cocone}(\mathcal{B},\mathcal{X}\ast \mathcal{A})$ by \cite[Lemma 2.5]{AT22}.
	
	Since a cotorsion pair is uniquely determined by either of its components, the two maps are clearly mutually inverse. They are also isomorphisms of posets by definition.
	
	Let $(\mathcal{U},\mathcal{V})\in \mathsf{scotors}(-\infty,x]$. We have $\mathbb{E}_{\mathcal{Y}}^{2}(\mathcal{U}\cap \mathcal{Y},\mathcal{V})\cong \mathbb{E}^{2}(\mathcal{U}\cap \mathcal{Y},\mathcal{V})=0$. Let $(\mathcal{A},\mathcal{B})\in \mathsf{scotors}\,\mathcal{Y}$. Since $\mathbb{E}^{2}(\mathcal{A},\mathcal{B})\cong \mathbb{E}_{\mathcal{Y}}^{2}(\mathcal{A},\mathcal{B})=0$, together with $\mathbb{E}^{2}(\mathcal{X},\mathcal{B})=0$, we obtain $\mathbb{E}^{2}(\mathcal{X}\ast \mathcal{A},\mathcal{B})=0$. By Lemma \ref{scotorsion}, we have $\mathsf{scotors}\,\mathcal{Y}=\mathsf{hcotors}\,\mathcal{Y}$.
	
	(2) By a similar argument as in (1), we obtain $\mathbb{E}_{\mathcal{Y}}^{m}(T,S)\cong \mathbb{E}^{m}(T,S)$ for any $T,S\in \mathcal{Y}$ and $m\geq 1$ since $\mathbb{E}^{\geq 1}(\mathcal{X},\mathcal{Y})=0$. Let $(\mathcal{U},\mathcal{V})\in \mathsf{scotors}(-\infty,x]$. Then $(\mathcal{U},\mathcal{V})=(\mathsf{add}(\mathcal{X}\ast \mathcal{A}),\mathcal{B})$ for some $(\mathcal{A},\mathcal{B})\in \mathsf{hcotors}\,\mathcal{Y}$. Thus for any $m\geq 1$, $\mathbb{E}^{m}(\mathcal{A},\mathcal{B})\cong \mathbb{E}_{\mathcal{Y}}^{m}(\mathcal{A},\mathcal{B})=0$. Combining with $\mathbb{E}^{m}(\mathcal{X},\mathcal{B})=0$, we obtain $\mathbb{E}^{m}(\mathcal{U},\mathcal{V})=0$. The last assertion is obvious.
	
	(3) By Lemma \ref{bddscotors}, $\mathcal{Y}=\mathcal{N}^{\wedge}$. Thus for any $T\in \mathcal{Y}$, there are $\mathbb{E}$-triangles
	\[N_{n}\rightarrowtail N_{n-1}\twoheadrightarrow T_{n-1}\dashrightarrow\]
	\[T_{n-1}\rightarrowtail N_{n-2}\twoheadrightarrow T_{n-2}\dashrightarrow\]
	\[\cdots\]
	\[T_{1}\rightarrowtail N_{0}\twoheadrightarrow T_{0}\dashrightarrow\]
	with $N_{i}\in \mathcal{N}$ and $T_{0}=T$. Thus $T_{i}\in \mathcal{N}^{\wedge}=\mathcal{Y}$. These are also $\mathbb{E}_{\mathcal{Y}}$-triangles, hence ${\rm pd}_{\mathcal{Y}}T< \infty$. Let $(\mathcal{U},\mathcal{V})\in \mathsf{bdd\text{-}cotors}(-\infty,x]$. Note $\mathcal{V}^{\vee}\cap \mathcal{Y}=\mathcal{Y}$ equals to $\mathcal{V}^{\vee}$ in $\mathcal{Y}$ since $\mathcal{Y}$ is closed under cones. In $\mathcal{Y}$, $\mathcal{Y}=\mathcal{N}^{\wedge}\subseteq (\mathcal{U}\cap \mathcal{Y})^{\wedge}$. Thus $(\mathcal{U}\cap \mathcal{Y},\mathcal{V})\in \mathsf{bdd\text{-}cotors}\,\mathcal{Y}$.
	
	(4) By (1) and (2), we have bijections $\mathsf{hcotors}(-\infty,x]\leftrightarrow \mathsf{hcotors}\,\mathcal{Y}$. The map from left to right sends bounded cotorsion pairs to bounded ones by (3). It suffices to check the other map. Assume $(\mathcal{U}\cap \mathcal{Y},\mathcal{V})\in \mathsf{bdd\text{-}hcotors}\,\mathcal{Y}$, then $\mathcal{Y}\subseteq \mathcal{V}^{\vee}$ in $\mathcal{C}$. Let $\mathcal{W}:=\mathcal{U}\cap \mathcal{V}$. Then $\mathcal{V},\mathcal{W}$ satisfy the conditions in \cite[Proposition 4.8]{AT22}. Thus $\mathcal{V}^{\vee}$ is closed under cocones. Hence $\mathcal{C}=\mathcal{Y}^{\vee}\subseteq \mathcal{V}^{\vee}$. Since $\mathcal{C}=\mathcal{X}^{\wedge}\subseteq \mathcal{U}^{\wedge}$, $(\mathcal{U},\mathcal{V})$ is bounded.
\end{proof}

Dually, we have bijections between cotorsion pairs greater than $x$ and those in $\mathcal{X}$. By Proposition \ref{first prop} and its dual, we recover \cite[Theorem 3.6]{AT22} and obtain some other results on intervals of cotorsion pairs.

\begin{cor}\label{main cor}
	Let $x_{i}=(\mathcal{X}_{i},\mathcal{Y}_{i})\in \mathsf{scotors}\,\mathcal{C},i=1,2$ such that $x_{1}\leq x_{2}$. Denote by $\mathcal{C}'$ (resp. $\mathcal{M},~\mathcal{N}$) the intersection $\mathcal{X}_{1}\cap \mathcal{Y}_{2}$ (resp. $\mathcal{X}_{1}\cap \mathcal{Y}_{1},~\mathcal{X}_{2}\cap \mathcal{Y}_{2}$). 
	
	(1) $\mathcal{C}'$ is an extriangulated category with enough projectives $\mathcal{N}$ and enough injectives $\mathcal{M}$. For any $T,S\in \mathcal{C}'$, $\mathbb{E}^{2}(T,S)\cong \mathbb{E}_{\mathcal{C}'}^{2}(T,S)$. There are mutually inverse isomorphisms of posets
	\begin{align*}
		\mathsf{cotors}\,[x_{1},x_{2}] & \leftrightarrow \mathsf{cotors}\,\mathcal{C}' \\
		(\mathcal{X},\mathcal{Y}) & \mapsto (\mathcal{X}\cap \mathcal{Y}_{2},\mathcal{X}_{1}\cap \mathcal{Y}) \\
		(\mathsf{add}(\mathcal{X}_{2}\ast \mathcal{A}),\mathsf{add}(\mathcal{B}\ast \mathcal{Y}_{1})) & \mapsfrom (\mathcal{A},\mathcal{B})
	\end{align*}
    which restrict to isomorphisms $\mathsf{scotors}\,[x_{1},x_{2}] \leftrightarrow \mathsf{hcotors}\,\mathcal{C}'$. If moreover $x_{1}$ and $x_{2}$ are bounded in $\mathcal{C}$, then $\mathcal{C}'=\mathcal{M}^{\vee}\cap \mathcal{N}^{\wedge}$ and for each $T\in \mathcal{C}'$, we have ${\rm pd}_{\mathcal{C}'}T<\infty$, ${\rm id}_{\mathcal{C}'}T<\infty$.  The isomorphisms become 
    \[\mathsf{bdd\text{-}cotors}\,[x_{1},x_{2}]\leftrightarrow \mathsf{bdd\text{-}cotors}\,\mathcal{C}'\]
    and restrict to isomorphisms $\mathsf{bdd\text{-}scotors}\,[x_{1},x_{2}]\leftrightarrow \mathsf{bdd\text{-}hcotors}\,\mathcal{C}'$.
	
	(2)  Assume $x_{1},x_{2}\in \mathsf{hcotors}\,\mathcal{C}$, then $\mathbb{E}^{m}(T,S)\cong \mathbb{E}_{\mathcal{C}'}^{m}(T,S)$ for any $T,S\in \mathcal{C}'$ and $m\geq 1$, and $\mathsf{presilt}\,\mathcal{C}'=\{\mathcal{T}\in \mathsf{presilt}\,\mathcal{C}\,|\,\mathcal{T}\subseteq \mathcal{C}'\}$. The restrictions in (1) become  $\mathsf{hcotors}\,[x_{1},x_{2}]\leftrightarrow \mathsf{hcotors}\,\mathcal{C}'$. If moreover $x_{1}$ and $x_{2}$ are bounded in $\mathcal{C}$, then $\mathcal{C}'={^{\bot}\mathcal{M}}\cap \mathcal{N}^{\bot}$ and the restrictions become 
	\[\mathsf{bdd\text{-}hcotors}\,[x_{1},x_{2}]\leftrightarrow \mathsf{bdd\text{-}hcotors}\,\mathcal{C}'.\]
\end{cor}
\begin{proof}
	(1) By applying Proposition \ref{first prop} (1), then its dual and conversely, we obtain immediately the first statement. We also obtain $\mathbb{E}^{2}(T,S)\cong \mathbb{E}_{\mathcal{Y}_{2}}^{2}(T,S)\cong \mathbb{E}_{\mathcal{C}'}^{2}(T,S)$ for any $T,S\in \mathcal{C}'$. By Proposition \ref{first prop} (1), we have isomorphisms of posets $\mathsf{cotors}\,[x_{1},x_{2}]\leftrightarrow \mathsf{cotors}_{\mathcal{Y}_{2}}[x_{1}',+\infty)$ where $x_{1}'=(\mathcal{X}_{1}\cap \mathcal{Y}_{2},\mathcal{Y}_{1})\in \mathsf{scotors}\,\mathcal{Y}_{2}$ and the subscript $\mathcal{Y}_{2}$ means considering cotorsion pairs of the extriangulated category $\mathcal{Y}_{2}$. Then by its dual, we have isomorphisms of posets $\mathsf{cotors}_{\mathcal{Y}_{2}}[x_{1}',+\infty)\leftrightarrow \mathsf{cotors}\,\mathcal{C}'$. Composition of these maps are exactly the maps given. The restriction follows simultaneously. Assume moreover $x_{1}$ and $x_{2}$ are bounded. By Proposition \ref{first prop} (3), $(\mathcal{X}_{1}\cap \mathcal{Y}_{2},\mathcal{Y}_{1})\in \mathsf{bdd\text{-}scotors}\,\mathcal{Y}_{2}$. By its dual, for any $T\in \mathcal{C}'$, ${\rm id}_{\mathcal{C}'}T<\infty$. The other half is similar. It implies $\mathsf{cotors}\,\mathcal{C}'=\mathsf{bdd\text{-}cotors}\,\mathcal{C}'$. $\mathcal{C}'=\mathcal{M}^{\vee}\cap \mathcal{N}^{\wedge}$ follows from Lemma \ref{bddscotors}. Since $x_{1}$ and $x_{2}$ are bounded, each cotorsion pair in the interval is bounded automatically. Thus we finish the proof.
	
	(2) For any $T,S\in \mathcal{C}'$ and $m\geq 1$, we have $\mathbb{E}^{m}(T,S)\cong \mathbb{E}_{\mathcal{Y}_{2}}^{m}(T,S)\cong \mathbb{E}_{\mathcal{C}'}^{m}(T,S)$ by Proposition \ref{first prop} (1), (2) and its dual. By (1) and Proposition \ref{first prop} (2), we have $\mathsf{hcotors}\,[x_{1},x_{2}]\leftrightarrow \mathsf{hcotors}\,\mathcal{C}'$. Assume moreover $x_{1},x_{2}$ are bounded. Then $\mathcal{C}'={^{\bot}\mathcal{M}}\cap \mathcal{N}^{\bot}$ follows from Theorem \ref{AT22thm5.7}. The last statement follows from (1).
\end{proof}

Next we give a bijection between intervals of silting subcategories in $\mathcal{C}$ and silting subcategories in extension-closed subcategories (as extriangulated categories), which we call silting interval reduction.

\begin{thm}\label{main thm_1}
	Let $\mathcal{M},\mathcal{N}$ be silting subcategories of $\mathcal{C}$ such that $\mathcal{M}\leq \mathcal{N}$. Then the map $\mathcal{T}\mapsto \mathcal{T}$ induces isomorphisms of posets
	\[\mathsf{silt}(-\infty,\mathcal{N}]=\{\mathcal{T}\in \mathsf{silt}\,\mathcal{C}\,|\,\mathcal{T}\subseteq \mathcal{N}^{\bot}\}\cong \mathsf{silt}(\mathcal{N}^{\bot}),\]
	\[\mathsf{silt}[\mathcal{M},+\infty)=\{\mathcal{T}\in \mathsf{silt}\,\mathcal{C}\,|\,\mathcal{T}\subseteq {^{\bot}\mathcal{M}}\}\cong \mathsf{silt}({^{\bot}\mathcal{M}}),\]
	\[\mathsf{silt}\,[\mathcal{M},\mathcal{N}]=\{\mathcal{T}\in \mathsf{silt}\,\mathcal{C}\,|\,\mathcal{T}\subseteq {^{\bot}\mathcal{M}}\cap \mathcal{N}^{\bot}\}\cong \mathsf{silt}({^{\bot}\mathcal{M}}\cap \mathcal{N}^{\bot}).\]
\end{thm}
\begin{proof}
	By Theorem \ref{AT22thm5.7}, Proposition \ref{first prop} (4) and its dual, we obtain the first two isomorphisms. The last one follows from Corollary \ref{main cor} (2) and Theorem \ref{AT22thm5.7}.
\end{proof}

If $\mathcal{C}$ is a triangulated category, we obtain a generalization of \cite[Theorem 2.3]{IJY}, \cite[Theorem 2.2]{PZ}, \cite[Corollary 1.22]{BZ} and \cite[Corollary 5.10]{GNP23}. Denote by $\mathsf{co\text{-}t\text{-}str}\,\mathcal{C}$ the set of co-t-structures in $\mathcal{C}$ and by $\mathsf{bdd\text{-}co\text{-}t\text{-}str}\,\mathcal{C}$ the bounded ones. Note that in triangulated categories, (bounded) hereditary cotorsion pairs are precisely (bounded) co-t-structures (cf. \cite[Lemma 3.3, Corollary 5.11]{AT22}). Let $\mathcal{R}=\mathsf{add}\mathcal{R}$ be a rigid subcategory in $\mathcal{C}$, the subcategory $\mathcal{R}\ast \mathcal{R}[1]\ast \cdots \ast \mathcal{R}[n]$ is called $(n+1)$-{\em term subcategory}. If $\mathcal{R}$ is silting, then a silting subcategory $\mathcal{S}$ is called $(n+1)$-{\em term silting} if $\mathcal{S}\subseteq \mathcal{R}\ast \mathcal{R}[1]\ast \cdots \ast \mathcal{R}[n]$.

\begin{cor}\label{generalize PZ}
	Assume $\mathcal{C}$ is a triangulated category with a shift functor $[1]$. Let $(\mathcal{X},\mathcal{Y})$ be a co-t-structure, $\mathcal{S}:=\mathcal{X}\cap \mathcal{Y}$ and $\mathcal{C}':=\mathcal{X}[n]\cap \mathcal{Y}$ ($n\geq 1$).
	Then 
 
	(1) $\mathcal{C}'=\mathcal{S}\ast \mathcal{S}[1]\ast \cdots \ast \mathcal{S}[n]$ and there are mutually inverse isomorphisms of posets
	\[\{(\mathcal{U},\mathcal{V})\in \mathsf{co\text{-}t\text{-}str}\,\mathcal{C}\,|\,\mathcal{X}\subseteq \mathcal{U}\subseteq \mathcal{X}[n]\}\leftrightarrow \mathsf{hcotors}\,\mathcal{C}'\leftrightarrow \mathsf{silt}\,\mathcal{C}'\]
	where the maps are given by Corollary \ref{main cor} (1) and Theorem \ref{AT22thm5.7} respectively.
	
	(2) If $(\mathcal{X},\mathcal{Y})$ is bounded, or equivalently $\mathcal{S}$ is silting (see \cite[Corollary 5.9]{MSSS}). Then $(n+1)$-term silting subcategories of $\mathcal{C}$ are precisely silting subcategories of the extriangulated category $\mathcal{C}'$ (i.e. $\{\mathcal{T}\in \mathsf{silt}\,\mathcal{C}\,|\,\mathcal{T}\subseteq \mathcal{C}'\}\cong \mathsf{silt}\,\mathcal{C}'$). Moreover there is a commutative diagram of bijections
	\[\begin{tikzcd}
		\{(\mathcal{U},\mathcal{V})\in \mathsf{bdd\text{-}co\text{-}t\text{-}str}\,\mathcal{C}\,|\,\mathcal{X}\subseteq \mathcal{U}\subseteq \mathcal{X}[n]\} \arrow[r, leftrightarrow] \arrow[d, leftrightarrow] & \mathsf{hcotors}\,\mathcal{C}' \arrow[d, leftrightarrow] \\
		\{\mathcal{T}\in \mathsf{silt}\,\mathcal{C}\,|\,\mathcal{T}\subseteq \mathcal{C}'\} \arrow[r,equal,"\sim"] & \mathsf{silt}\,\mathcal{C}'
	\end{tikzcd}\]
    where the vertical maps are given by Theorem \ref{AT22thm5.7}.
\end{cor}
\begin{proof}
	(1) Since $\mathcal{S}[i]=\mathcal{X}[i]\cap \mathcal{Y}[i]\subseteq \mathcal{C}'$ for $0\leq i\leq n$, $\mathcal{S}\ast \mathcal{S}[1]\ast \cdots \ast \mathcal{S}[n]\subseteq \mathcal{C}'$. For each $T\in \mathcal{C}'$, consider a triangle $X_{n-1}\rightarrow T\rightarrow Y_{n}\rightarrow X_{n-1}[1]$ with $X_{n-1}\in \mathcal{X}[n-1],Y_{n}\in \mathcal{Y}[n]$. Then $Y_{n}\in \mathcal{S}[n]$ and $X_{n-1}\in \mathcal{X}[n-1]\cap \mathcal{Y}$. By induction, we obtain $T\in \mathcal{S}\ast \mathcal{S}[1]\ast \cdots \ast \mathcal{S}[n]$. Then $\mathcal{C}'=\mathcal{S}\ast \mathcal{S}[1]\ast \cdots \ast \mathcal{S}[n]$.
	Since $\mathcal{C}'$ has enough projectives (resp. injectives) $\mathcal{S}$ (resp. $\mathcal{S}[n]$) and $\mathcal{C}'=\mathcal{S}_{n}^{\wedge}=\mathcal{S}[n]_{n}^{\vee}$, cotorsion pairs in $\mathcal{C}'$ are bounded. By Corollary \ref{main cor} (2) and Theorem \ref{AT22thm5.7}, the result follows.
	
	(2) The result follows from Corollary \ref{main cor} (2) and Theorem \ref{AT22thm5.7}. The isomorphism $\{\mathcal{T}\in \mathsf{silt}\,\mathcal{C}\,|\,\mathcal{T}\subseteq \mathcal{C}'\}\cong \mathsf{silt}\,\mathcal{C}'$ is just part of Theorem \ref{main thm_1}.
\end{proof}

\begin{rem}
	(1) If $n=1$ in Corollary \ref{generalize PZ} (1), it is just \cite[Theorem 2.2]{PZ} and \cite[Corollary 5.10]{GNP23}. Because we have  $\mathsf{cotors}\,\mathcal{C}'=\mathsf{hcotors}\,\mathcal{C}'$ in this case.
	
	(2) If $n=1$, then $\{\mathcal{T}\in \mathsf{silt}\,\mathcal{C}\,|\,\mathcal{T}\subseteq \mathcal{C}'\}\cong \mathsf{silt}\,\mathcal{C}'$ in Corollary \ref{generalize PZ} (2) generalizes \cite[Theorem 3.4(2)]{FGL} by \cite[Theorem 4.3]{GNP23}.
\end{rem}

If the extriangulated categories are exact categories, then we obtain a reduction technique for tilting subcategories. First recall the notion of tilting subcategories in exact categories \cite{Sauter}. It generalizes tilting modules over Artin algebras and unifies many existing notions in various contexts.

\begin{defn}\cite[Definition 4.1]{Sauter}
	Let $\mathcal{E}$ be an exact category and $n$ be a non-negative integer. A subcategory $\mathcal{T}$ of $\mathcal{E}$ is {\em n-tilting} if it satisfies: 
	
	(1) $\mathcal{T}^{\bot}$ has enough projectives $\mathcal{T}$,
	
	(2) $(\mathcal{T}^{\bot})_{n}^{\vee}=\mathcal{E}$.
\end{defn}

Let $\mathsf{tilt}\,\mathcal{E}:=\{\mathcal{T}\subseteq \mathcal{E}\,|\,\mathcal{T}\text{ is {\em n}-tilting for some }n\}$. For $\mathcal{T}_{1}, \mathcal{T}_{2}\in \mathsf{tilt}\,\mathcal{E}$, write $\mathcal{T}_{1}\leq \mathcal{T}_{2}$ if $\mathcal{T}_{1}\subseteq \mathcal{T}_{2}^{\bot}$. Then $(\mathsf{tilt}\,\mathcal{E},\leq)$ is a poset, see \cite[Proposition 5.7]{Sauter}. 

\begin{lem}\label{silt=tilt}
	Let $\mathcal{E}$ be an essentially small exact category and $\mathcal{P}$ be the projectives. Assume $\mathcal{E}=\mathcal{P}^{\wedge}$. Then $\mathsf{tilt}\,\mathcal{E}=\{\mathcal{T}\in \mathsf{silt}\,\mathcal{E}\,|\,\mathcal{T}\subseteq \mathcal{P}_{n}^{\wedge}\text{ for some }n\geq 0\}$ as posets.
\end{lem}
\begin{proof}
	The inclusion "$\subseteq$" follows from \cite[Theorem 5.3]{Sauter} and \cite[Proposition 5.5]{AT22}. Let $\mathcal{T}\in \mathsf{silt}\,\mathcal{E}$ such that $\mathcal{T}\subseteq \mathcal{P}_{n}^{\wedge}$ for some $n\geq 0$. By \cite[Proposition 5.5]{AT22} and dimension shift, we obtain $\mathcal{P}\subseteq \mathcal{T}_{n}^{\vee}$. By \cite[Theorem 5.3]{Sauter}, $\mathcal{T}$ is $n$-tilting.
\end{proof}

\begin{rem}
    Under the condition of Lemma \ref{silt=tilt}, if there is a tilting object in $\mathcal{E}$ or $\mathcal{E}=\mathcal{P}_{n}^{\wedge}$ for some $n$, then $\mathsf{tilt}\,\mathcal{E}=\mathsf{silt}\,\mathcal{E}$ by \cite[Proposition 5.4, 5.5]{AT22}.
\end{rem}

\begin{prop}\label{appli in exact cat}
	Assume $\mathcal{C}$ is an exact category with enough projectives $\mathcal{P}$. Let $\mathcal{M},\mathcal{N}\in \mathsf{tilt}\,\mathcal{C}$ such that $\mathcal{M}\leq \mathcal{N}$ and $\mathcal{C}':={^{\bot}\mathcal{M}}\cap \mathcal{N}^{\bot}\cap \mathcal{P}^{\wedge}$. Then $\mathcal{T}\mapsto \mathcal{T}$ induces isomorphisms of posets $\mathsf{tilt}\,[\mathcal{M},\mathcal{N}]\cong \mathsf{tilt}\,\mathcal{C}'$.
\end{prop}
\begin{proof}
	Let $\mathcal{E}:=\mathcal{P}^{\wedge}$. Then $\mathcal{E}$ is a thick subcategory of $\mathcal{C}$ containing $\mathcal{P}$. Thus for any $X,Y\in \mathcal{E}$ and $k\geq 1$, we have ${\rm Ext}_{\mathcal{C}}^{k}(X,Y)\cong {\rm Ext}_{\mathcal{E}}^{k}(X,Y)$. By \cite[Lemma 4.8, Corollary 5.10]{Sauter}, the map $\mathcal{T}\mapsto \mathcal{T}$ induces $\mathsf{tilt}\,\mathcal{C}\cong \mathsf{tilt}\,\mathcal{E}$ as posets. By Lemma \ref{silt=tilt}, we have 
	\[\mathsf{tilt}\,\mathcal{E}=\{\mathcal{T}\in \mathsf{silt}\,\mathcal{E}\,|\,\mathcal{T}\subseteq \mathcal{P}_{n}^{\wedge}\text{ for some }n\geq 0\}.\]
	Since $\mathcal{C}'={^{\bot}\mathcal{M}}\cap \mathcal{N}^{\bot}\cap \mathcal{P}^{\wedge}=({^{\bot}\mathcal{M}}\cap \mathcal{E})\cap (\mathcal{N}^{\bot}\cap \mathcal{E})={^{\bot_{\mathcal{E}}}\mathcal{M}}\cap \mathcal{N}^{\bot_{\mathcal{E}}}$, in $\mathcal{C}'$ we have
	\[\mathsf{tilt}\,\mathcal{C}'=\{\mathcal{T}'\in \mathsf{silt}\,\mathcal{C}'\,|\,\mathcal{T}'\subseteq \mathcal{N}_{n}^{\wedge}\text{ for some }n\geq 0\}\]
	by Corollary \ref{main cor} (1) and Lemma \ref{silt=tilt}. It suffices to show
    \[\{\mathcal{T}\in \mathsf{silt}\,\mathcal{E}\,|\,\mathcal{M}\leq \mathcal{T}\leq \mathcal{N},\mathcal{T}\subseteq \mathcal{P}_{n}^{\wedge}\text{ for some }n\geq 0\}=\{\mathcal{T}'\in \mathsf{silt}\,\mathcal{C}'\,|\,\mathcal{T}'\subseteq \mathcal{N}_{n}^{\wedge}\text{ for some }n\geq 0\},\]
    which is a restriction of an isomorphism in Theorem \ref{main thm_1}. Since the higher extensions in $\mathcal{C}'$ coincide with those in $\mathcal{E}$ by Corollary \ref{main cor} (2), the inclusion $\subseteq$ is clear. Since ${\rm pd}_{\mathcal{C}}\mathcal{N}\leq m$ for some $m$ by \cite[Theorem 5.3]{Sauter}, the inclusion $\supseteq$ follows by dimension shift.
\end{proof}

Dually we may also consider reduction of intervals of cotilting subcategories in exact categories, but we omit the statements.

\section{Silting in 0-Auslander extriangulated categories}\label{section 4}

Let $\mathcal{C}=(\mathcal{C},\mathbb{E},\mathfrak{s})$ be an extriangulated category.

\begin{defn}\cite[Definition 3.5]{GNP23}
	Let $\mathcal{C}=(\mathcal{C},\mathbb{E},\mathfrak{s})$ be an extriangulated category with enough projectives. Its {\em dominant dimension} dom.dim$\,\mathcal{C}$ is defined to be the largest integer $n$ such that for any projective object $P$, there are $n$ $\mathbb{E}$-triangles
	\[P\rightarrowtail I_{0}\twoheadrightarrow M_{1}\dashrightarrow \]
	\[M_{1}\rightarrowtail I_{1}\twoheadrightarrow M_{2}\dashrightarrow \]
	\[\cdots\]
	\[M_{n-1}\rightarrowtail I_{n-1}\twoheadrightarrow M_{n}\dashrightarrow \]
	with $I_{k}$ being projective-injective for $0\leq k\leq n-1$. If such an $n$ does not exist, then let dom.dim$\,\mathcal{C}=\infty$. Dually we define {\em codominant dimension} codom.dim$\,\mathcal{C}$.
\end{defn}

\begin{defn}\cite[Definition 3.7]{GNP23}\label{defn of 0-Aus}
	$\mathcal{C}=(\mathcal{C},\mathbb{E},\mathfrak{s})$ is {\em 0-Auslander} if it has enough projectives and ${\rm pd}~\mathcal{C}\leq 1\leq {\rm dom.dim}\,\mathcal{C}$, or equivalently, if it has enough injectives and ${\rm id}~\mathcal{C}\leq 1\leq {\rm codom.dim}\,\mathcal{C}$. It is {\em reduced} if its projective-injective objects are 0.
\end{defn}

For 0-Auslander $\mathcal{C}$, denote by $\mathcal{P}$ (resp. $\mathcal{I}$) the subcategory of projectives (resp. injectives), then the extriangulated category  $\mathcal{C}/[\mathcal{P}\cap \mathcal{I}]$ is easily checked to be reduced 0-Auslander (cf. \cite[Proposition 3.30]{NP}). For convenience we mainly consider reduced 0-Auslander extriangulated categories. We collect some properties as follows (see \cite[Section 4.5]{PPPP23}).

\begin{lem}\label{properties of 0-Aus}
	Let $\mathcal{C}$ be reduced 0-Auslander with projectives $\mathcal{P}$ and injectives $\mathcal{I}$. Then
	
	(1) ${\rm Hom}_{\mathcal{C}}(\mathcal{P},\mathcal{I})=0$;
	
	(2) For each $X\in \mathcal{C}$, there is an $\mathbb{E}$-triangle $P\rightarrowtail X\twoheadrightarrow I\dashrightarrow $ with $P\in \mathcal{P}$ and $I\in \mathcal{I}$;
	
	(3) Any morphism $P\rightarrow X$ (resp. $X\rightarrow I$) with $P\in \mathcal{P}$ (resp. $I\in \mathcal{I}$) is an inflation (resp. deflation);
	
	(4) For each $P\in \mathcal{P}$ (resp. $I\in \mathcal{I}$), fix an $\mathbb{E}$-triangle $P\rightarrowtail 0\twoheadrightarrow \Sigma P \dashrightarrow$ (resp. $\Omega I \rightarrowtail 0\twoheadrightarrow I\dashrightarrow $). Then they define mutually quasi-inverse equivalence
	\[\Sigma:\mathcal{P}\rightleftarrows \mathcal{I}:\Omega;\]
	
	(5) The functor $F: \mathcal{C}\rightarrow \mathsf{mod}\mathcal{P}~X\mapsto {\rm Hom}_{\mathcal{C}}(-,X)|_{\mathcal{P}}$ induces an equivalence $\mathcal{C}/[\mathcal{I}]\stackrel{\sim}\rightarrow \mathsf{mod}\mathcal{P}$, where $\mathsf{mod}\mathcal{P}$ is the category of finitely presented contravariant functors.
\end{lem}

Throughout this section, we assume $\mathcal{C}$ is reduced 0-Auslander, $k$-linear ($k$ is a field), Hom-finite, Krull-Schmidt and $\mathcal{P}=\mathsf{add}P$ unless otherwise specified. Let $I=\Sigma P$, then $\mathcal{I}=\mathsf{add}I$. Then $P$ and $I$ are silting objects. By \cite[Proposition 5.4]{AT22}, we consider silting objects instead of silting subcategories in $\mathcal{C}$. Thus $\mathsf{silt}\,\mathcal{C}$ denote the set of isomorphism classes of basic silting objects in $\mathcal{C}$. 

Let $A:={\rm End}_{\mathcal{C}}(P)$. The functor $F$ in Lemma \ref{properties of 0-Aus} (5) is as follows: 

 $$\overline{(-)}:={\rm Hom}_{\mathcal{C}}(P,-):\mathcal{C}\rightarrow \mathsf{mod}A.$$  It induces an equivalence $\mathcal{C}/[\mathcal{I}]\stackrel{\sim}\rightarrow \mathsf{mod}A$.

For an additive subcategory $\mathcal{T}$ closed under direct summands, denote by $|\mathcal{T}|$ the number of isomorphism classes of indecomposable objects in $\mathcal{T}$. For an object $X$,  $|X|:=|\mathsf{add}X|$.

\subsection{Bijections}
For any $X\in \mathcal{C}$, assume $X\cong X'\oplus I_{X}$ where $I_{X}$ is the maximal direct summand of $X$ which is injective. Then $X\mapsto (\overline{X},\overline{\Omega I_{X}})$ and the inverse map give bijections $\text{iso}\,\mathcal{C}\leftrightarrow \text{iso}(\mathsf{mod}A) \times \text{iso}(\mathsf{proj}A)$. We show the bijections restrict to bijections between silting objects and support $\tau$-tilting pairs in $\mathsf{mod}A$ (see \cite{AIR}). Note that in \cite[Remark 5.15]{GNP23}, the authors pointed out without a proof that there is a bijection between cotorsion pairs in $\mathcal{C}$ and functorially finite torsion classes in $\mathsf{mod}A$ as in \cite[Theorem 3.6]{PZ}. For completeness and convenience of the readers, we provide proofs in this subsection. First recall a condition (WIC) in an arbitrary extriangulated category $\mathcal{C}$.

{\bf (WIC)} For any morphisms $f:X\rightarrow Y$ and $g:Y\rightarrow Z$. If $gf$ is an inflation, so is $f$. Dually if $gf$ is a deflation, so is $g$.

Note that an extriangulated category satisfies (WIC) if and only if it is weakly idempotent complete (see \cite[Proposition 2.7]{Klapproth}). In particular, Krull-Schmidt categories are weakly idempotent complete.

\begin{defn}\cite[Definition 0.3]{AIR}
	Let $(M,P)$ be a pair of modules with $M\in \mathsf{mod}A$ and $P\in \mathsf{proj}A$.
	
	(1) $(M,P)$ is called a {\em $\tau$-rigid} pair if $M$ is $\tau$-rigid (i.e. ${\rm Hom}_{A}(M,\tau M)=0$) and ${\rm Hom}_{A}(P,M)=0$.
	
	(2) $(M,P)$ is called a {\em support $\tau$-tilting} pair if it is $\tau$-rigid and $|M|+|P|=|A|$. In this case, $M$ is called a support $\tau$-tilting module. If $(M,0)$ is support $\tau$-tilting, then $M$ is called a $\tau$-{\em tilting} module.

\end{defn}

Denote by $s\tau\text{-}\mathsf{tilt}A$ the set of basic support $\tau$-tilting pairs in $\mathsf{mod}A$. Sometimes $ s\tau\text{-}\mathsf{tilt}A$ also denotes the set of basic support $\tau$-tilting modules. Let $\tau\text{-}\mathsf{tilt}A$ (resp. $\mathsf{tilt}A$) denote the set of basic $\tau$-tilting (resp. 1-tilting) modules.

\begin{prop}\label{presilting-rigid}
	$X$ is presilting if and only if $(\overline{X},\overline{\Omega I_{X}})$ is a $\tau$-rigid pair.
\end{prop}
\begin{proof}
	Take an $\mathbb{E}$-triangle $P_{1}'\stackrel{f}\rightarrowtail P_{0}'\stackrel{g}\twoheadrightarrow X'\dashrightarrow$. We decompose $g$ as $(g',0)$ such that $g'$ is right minimal. By (WIC), $g'$ is a deflation. Thus we may assume in the beginning that $g$ is right minimal. We claim $\overline{P_{1}'}\stackrel{\overline{f}}\rightarrow \overline{P_{0}'}\stackrel{\overline{g}}\rightarrow \overline{X}\rightarrow 0$ is a minimal projective presentation. Clearly $\overline{g}$ is right minimal. Since $X'$ has no injective direct summand, $\overline{f}$ is right minimal. ${\rm Hom}_{A}(\overline{P_{0}'},\overline{X})\stackrel{\overline{f}^{\ast}}\rightarrow {\rm Hom}_{A}(\overline{P_{1}'},\overline{X})$ is surjective if and only if ${\rm Hom}_{\mathcal{C}}(P_{0}',X)\stackrel{f^{\ast}}\rightarrow {\rm Hom}_{\mathcal{C}}(P_{1}',X)$ is surjective if and only if $\mathbb{E}(X',X)=0$. By \cite[Proposition 2.4]{AIR}, $\overline{X}$ is $\tau$-rigid if and only if $\mathbb{E}(X',X)=0$. Considering the $\mathbb{E}$-triangle $\Omega I_{X}\rightarrowtail 0 \twoheadrightarrow I_{X}\dashrightarrow$, we have ${\rm Hom}_{A}(\overline{\Omega I_{X}},\overline{X})=0$ if and only if ${\rm Hom}_{\mathcal{C}}(\Omega I_{X},X)=0$ if and only if $\mathbb{E}(I_{X},X)=0$. The proof is complete.
\end{proof}

Denote by $\mathsf{silt}^{\mathcal{I}}\mathcal{C}$ the subset of $\mathsf{silt}\,\mathcal{C}$ whose elements have no nonzero direct summands in $\mathcal{I}$. Let $\mathsf{silt}^{\mathcal{I}}_{0}\mathcal{C}:=\{X\in \mathsf{silt}\,\mathcal{C}\,|\,[\mathsf{add}X](I,I)=0\}$. Note that we use the superscript $\mathcal{I}$ in order to distinguish between $\mathsf{silt}^{\mathcal{I}}\mathcal{C}$ and $\mathsf{silt}_{U}\mathcal{C}$ that appears later.

\begin{cor}\label{silt-tau-tilt}
	$\overline{(-)}$ induces a bijection  $\mathsf{silt}\,\mathcal{C}\rightarrow s\tau\text{-}\mathsf{tilt}A$, which restricts to bijections $\mathsf{silt}^{\mathcal{I}}\mathcal{C}\rightarrow \tau\text{-}\mathsf{tilt}A$ and $\mathsf{silt}^{\mathcal{I}}_{0}\mathcal{C}\rightarrow \mathsf{tilt}A$.
\end{cor}
\begin{proof}
	By \cite[Theorem 4.3]{GNP23}, silting is equivalent to maximal rigid in 0-Auslander categories. Thus the first one is obvious by Proposition \ref{presilting-rigid}. The second one follows from construction. For a silting object $X\in \mathcal{C}$, consider an $\mathbb{E}$-triangle $P_{1}\stackrel{f}\rightarrowtail P_{0}\stackrel{g}\twoheadrightarrow X\stackrel{\delta}\dashrightarrow $ with $g$ right minimal, then $\overline{X}$ is tilting if and only if $\overline{f}$ is a monomorphism. Since $g$ is an inflation, there is an $\mathbb{E}$-triangle $P_{0}\stackrel{g}\rightarrowtail X\stackrel{h}\twoheadrightarrow I_{0}\dashrightarrow$ with $I_{0}=\Sigma P_{1}\in \mathcal{I}$. One can check $[\mathsf{add}X](I,I)=0$ if and only if any morphism $I\rightarrow X$ factors through $g$. Consider a morphism of $\mathbb{E}$-triangles
	\[\begin{tikzcd}
		P \arrow[r,tail] \arrow[d,"a"] & 0 \arrow[r,two heads] \arrow[d,"0"] & I \arrow[r,dashed,"\alpha"] \arrow[d,"c"] & {} \\
		P_{1} \arrow[r,tail,"f"] & P_{0} \arrow[r,two heads,"g"] & X \arrow[r,dashed,"\delta"] & {}
	\end{tikzcd}.\]
    Assume $\overline{f}$ is a monomorphism. For any $c:I\rightarrow X$, there is $a:P\rightarrow P_{1}$ such that $(a,0,c)$ is a morphism of $\mathbb{E}$-triangles. Then $a=0$. Hence $c^{\ast}\delta=a_{\ast}\alpha=0$ and $c$ factors through $g$. On the other hand, assume any morphism $I\rightarrow X$ factors through $g$. Let $a:P\rightarrow P_{1}$ such that $fa=0$, then there is $c:I\rightarrow X$ such that $(a,0,c)$ is a morphism of $\mathbb{E}$-triangles. Then $a_{\ast}\alpha=c^{\ast}\delta=0$. Hence $a$ factors through 0 and $a=0$. The third one now follows.
\end{proof}

Corollary \ref{silt-tau-tilt} generalizes many well-known results (e.g. \cite{AIR}, \cite{IJY}, \cite{YZ}, \cite{FGL}). Next we show there are bijections between  cotorsion pairs in $\mathcal{C}$ and left weak cotorsion-torsion triples in $\mathsf{mod}A$, which generalize \cite{PZ} and \cite{BZ}. Recall that for an additive subcategory $\mathcal{X}\subseteq \mathsf{mod}A$, define
\[\mathsf{Fac}\mathcal{X}:=\{M\in \mathsf{mod}A \,|\,\exists \text{ epimorphism } X\rightarrow M, X\in \mathcal{X}\}.\]
We say $\mathcal{X}$ is {\em factor closed} if $\mathsf{Fac}\mathcal{X}\subseteq \mathcal{X}$. For an object $M$, $\mathsf{Fac}M:=\mathsf{Fac}(\mathsf{add}M)$.

\begin{defn}\cite[Definition 0.2]{BZ}
    A pair $(\mathcal{D},\mathcal{T})$ of subcategories of $\mathsf{mod}A$ is called a {\em left weak cotorsion pair} if the subcategories $\mathcal{D}$ and $\mathcal{T}$ are closed under  direct summands, ${\rm Ext}_{A}^{1}(\mathcal{D},\mathcal{T})=0$ and for each $M\in \mathsf{mod}A$, there are exact sequences
    \[0\rightarrow Y\rightarrow X\rightarrow M\rightarrow 0\text{ and }M\stackrel{f}\rightarrow Y'\rightarrow X'\rightarrow 0\]
    with $X,X'\in \mathcal{D}$ and $Y,Y'\in \mathcal{T}$ and $f$ a left $\mathcal{T}$-approximation.
\end{defn}

A triple $(\mathcal{D},\mathcal{T},\mathcal{F})$ of subcategories of $\mathsf{mod}A$ is called a {\em left weak cotorsion-torsion triple} if $(\mathcal{D},\mathcal{T})$ is a left weak cotorsion pair and $(\mathcal{T},\mathcal{F})$ is a torsion pair. Denote by $\mathsf{lw\text{-}cotors\text{-}tors}A$ the set of left weak cotorsion-torsion triples in $\mathsf{mod}A$. It is known that there is a bijection $s\tau\text{-}\mathsf{tilt}A\leftrightarrow \mathsf{lw\text{-}cotors\text{-}tors}A$, see \cite[Theorem 0.4]{BZ}.

\begin{rem}\label{uniquely determined}
    Note that a left weak cotorsion-torsion triple $(\mathcal{D},\mathcal{T},\mathcal{F})$ is uniquely determined by its mid term: the torsion class $\mathcal{T}$, which is functorially finite. 
\end{rem}

\begin{prop}\label{cotors-lwcotorstors}
	There are mutually inverse bijections 
	\begin{align*}
		\mathsf{cotors}\,\mathcal{C} & \leftrightarrow \mathsf{lw\text{-}cotors\text{-}tors}A \\
		(\mathcal{X},\mathcal{Y}) & \mapsto (\overline{\mathcal{X}},\overline{\mathcal{Y}},\overline{\mathcal{Y}}^{\bot_{0}}) \\
		({^{\bot_{1}}F^{-1}(\mathcal{T})},F^{-1}(\mathcal{T})) & \mapsfrom (\mathcal{D},\mathcal{T},\mathcal{F}).
	\end{align*}
\end{prop}
\begin{proof}
	Let $(\mathcal{X},\mathcal{Y})\in \mathsf{cotors}\,\mathcal{C}$. First show that $\overline{\mathcal{Y}}$ is factor closed. Let $Y\in \mathcal{Y}$ and $X\in \mathcal{C}$, assume that there is an epimorphism $\overline{Y}\stackrel{h}\rightarrow \overline{X}$. Choose an $f:Y\rightarrow X$ such that $\overline{f}=h$. Add a summand $I_{0}\in \mathcal{I}$ to $X$ such that $f'=\begin{bmatrix}
		f\\
		\ast
	\end{bmatrix}: Y\rightarrow X\oplus I_{0}$ is an inflation. Consider an $\mathbb{E}$-triangle $Y\stackrel{f'}\rightarrowtail X\oplus I_{0}\twoheadrightarrow Z\dashrightarrow$. Applying $\overline{(-)}$, we obtain $\overline{Z}=0$, hence $Z\in \mathcal{I}\subseteq \mathcal{Y}$. It implies $X\oplus I_{0}\in \mathcal{Y}$ and hence $X\in \mathcal{Y}$. 

    Clearly $\overline{\mathcal{X}}$ is closed under direct summands. For any $X\in \mathcal{X}$, take an $\mathbb{E}$-triangle $P_{1}\stackrel{a}\rightarrowtail P_{0}\twoheadrightarrow X\dashrightarrow$ with $P_{0},P_{1}\in \mathcal{P}$. For any $Y\in \mathcal{Y}$, then ${\rm Hom}_{\mathcal{C}}(a,Y)$ is surjective. Thus ${\rm Hom}_{A}(\overline{a},\overline{Y})$ is also surjective and hence ${\rm Ext}_{A}^{1}(\overline{X},\overline{Y})=0$.
    
    For any $T\in \mathcal{C}$, by definition, there are $\mathbb{E}$-triangles $Y_{1}\rightarrowtail X_{1}\stackrel{f}\twoheadrightarrow T\dashrightarrow$ and $T\stackrel{g}\rightarrowtail Y_{2}\twoheadrightarrow X_{2}\dashrightarrow$ with $X_{1},X_{2}\in \mathcal{X}$ and $Y_{1},Y_{2}\in \mathcal{Y}$. Applying $\overline{(-)}$, we obtain two exact sequences $\overline{Y_{1}}\rightarrow \overline{X_{1}}\stackrel{\overline{f}}\rightarrow \overline{T}\rightarrow 0$ and $\overline{T}\stackrel{\overline{g}}\rightarrow \overline{Y_{2}}\rightarrow \overline{X_{2}}\rightarrow 0$ such that $\overline{f},\overline{g}$ are approximations. Since ${\rm ker}\overline{f}\in \mathsf{Fac}\overline{Y_{1}}$ and $\overline{\mathcal{Y}}$ is factor closed, we have ${\rm ker}\overline{f}\in \overline{\mathcal{Y}}$.
    
    By the previous discussion, we obtain $\overline{\mathcal{Y}}=\mathsf{Fac}\overline{\mathcal{X}\cap \mathcal{Y}}$. Thus by the Horseshoe lemma, $\overline{\mathcal{Y}}$ is extension closed. By now, we have proved $(\overline{\mathcal{X}},\overline{\mathcal{Y}},\overline{\mathcal{Y}}^{\bot_{0}})$ is a left weak cotorsion-torsion triple. Since $\mathcal{I}\subseteq \mathcal{Y}$, $\mathcal{Y}$ is the preimage of $\overline{\mathcal{Y}}$. Thus the map $(\mathcal{X},\mathcal{Y})\mapsto (\overline{\mathcal{X}},\overline{\mathcal{Y}},\overline{\mathcal{Y}}^{\bot_{0}})$ is injective. 
    
    For a triple $(\mathcal{D},\mathcal{T},\mathcal{F})\in \mathsf{lw\text{-}cotors\text{-}tors}A$, there is a cotorsion pair $(\mathcal{U},\mathcal{V})$ in $\mathcal{C}$ such that the corresponding silting object $T\in \mathcal{C}$ satisfies $\mathcal{T}=\mathsf{Fac}\overline{T}$. Since $\overline{\mathcal{V}}=\mathsf{Fac}\,\overline{\mathcal{U}\cap \mathcal{V}}=\mathsf{Fac}\overline{T}=\mathcal{T}$, we obtain $(\overline{\mathcal{U}},\overline{\mathcal{V}},\overline{\mathcal{V}}^{\bot_{0}})=(\mathcal{D},\mathcal{T},\mathcal{F})$. Thus the map is also surjective. This completes the proof.
\end{proof}

\begin{rem}\label{commutative square}
	For $(\mathcal{D},\mathcal{T},\mathcal{F}),(\mathcal{D}',\mathcal{T}',\mathcal{F}')\in \mathsf{lw\text{-}cotors\text{-}tors}A$, we write $(\mathcal{D},\mathcal{T},\mathcal{F})\leq (\mathcal{D}',\mathcal{T}',\mathcal{F}')$ if $\mathcal{T}\leq \mathcal{T}'$. Then by Remark \ref{uniquely determined}, $\mathsf{lw\text{-}cotors\text{-}tors}A$ becomes a poset. Thus the bijections in Proposition \ref{cotors-lwcotorstors} are also isomorphisms of posets. They are compatible with $\mathsf{silt}\,\mathcal{C}\leftrightarrow s\tau\text{-}\mathsf{tilt}\,A$, which makes the latter ones isomorphisms of posets. In \cite[Theorem 1.3]{GNP23}, the authors proved a mutation theory on silting objects in $\mathcal{C}$. By their result, $X,Y\in \mathsf{silt}\,\mathcal{C}$ are mutations of each other if and only if $(\overline{X},\overline{\Omega I_{X}})$ and $(\overline{Y},\overline{\Omega I_{Y}})$ are mutations of each other.
\end{rem}

\subsection{Reduction}\label{Reduction}
In this subsection, we prove a reduction theorem (Theorem \ref{reduction in 0-Aus}) in the reduced 0-Auslander extriangulated category $\mathcal{C}$ as an application of silting interval reduction. Note that any presilting subcategory of $\mathcal{C}$ has an additive generator. Given a presilting object $U\in \mathcal{C}$ which is not silting. We define a subset $\mathsf{silt}_{U}\mathcal{C}$ of silting objects in $\mathcal{C}$ as $\mathsf{silt}_{U}\mathcal{C}:=\{T\in \mathsf{silt}\,\mathcal{C}\,|\,U\in \mathsf{add}T\}$. First recall the construction of Bongartz and co-Bongartz completions of $U$.

\begin{lem}\cite[Proposition 4.6]{GNP23}\label{two completions}
	Let $f:U'\rightarrow I$ (resp. $g:P\rightarrow U''$) be the right (resp. left) minimal $\mathsf{add}U$-approximation. It is a deflation (resp. inflation). Consider the $\mathbb{E}$-triangle $X\stackrel{a}\rightarrowtail U'\stackrel{f}\twoheadrightarrow I\dashrightarrow$ (resp. $P\stackrel{g}\rightarrowtail U''\stackrel{b}\twoheadrightarrow Y\dashrightarrow$). Then $X\oplus U$ (resp. $Y\oplus U$) is a silting object.
\end{lem}
\begin{proof}
	By Lemma \ref{properties of 0-Aus} (3), $f$ (resp. $g$) is a deflation (resp. inflation). The rest follows from \cite[Proposition 4.6]{GNP23} and its dual.
\end{proof}

We call $X\oplus U$ (resp. $Y\oplus U$) the {\em Bongartz} (resp. {\em co-Bongartz}) {\em completion} of $U$ and denote it by $N$ (resp. $M$).

\begin{lem}\label{proj-inj}
	$\mathsf{add}M\cap \mathsf{add}N=\mathsf{add}U$.
\end{lem}
\begin{proof}
	It suffices to show $X$ and $Y$ have no nonzero direct summands in common. Since there is an $\mathbb{E}$-triangle $P\rightarrowtail 0 \twoheadrightarrow I \dashrightarrow$, we obtain the following commutative diagrams by \cite[Proposition 3.15]{NP} and its dual.
	\[\begin{tikzcd}
		& P \arrow[r,equal] \arrow[d,tail,"c"] & P \arrow[d,tail] & \\
		X \arrow[r,tail,"1"] \arrow[d,equal] & X \arrow[r,two heads] \arrow[d,two heads,"a"] & 0 \arrow[r,dashed] \arrow[d,two heads] & {} \\
		X \arrow[r,tail,"a"] & U' \arrow[r,two heads,"f"] \arrow[d,dashed] & I \arrow[r,dashed] \arrow[d,dashed] & {} \\
		& {} & {} &
	\end{tikzcd} \qquad
      \begin{tikzcd}
      	P \arrow[r,tail,"g"] \arrow[d,tail] & U'' \arrow[r,two heads,"b"] \arrow[d,tail,"b"] & Y \arrow[r,dashed] \arrow[d,equal] & {} \\
      	0 \arrow[r,tail] \arrow[d,two heads] & Y \arrow[r,two heads,"1"] \arrow[d,two heads,"d"] & Y \arrow[r,dashed] & {} \\
      	I \arrow[r,equal] \arrow[d,dashed] & I \arrow[d,dashed] & & \\
      	{} & {} & &
      \end{tikzcd}\]
    For any $h:X\rightarrow Y$, we obtain a diagram
    \[\begin{tikzcd}
    	P \arrow[r,tail,"c"] \arrow[d,"t",swap] & X \arrow[r,two heads,"a"] \arrow[dl,dashed,"v",swap] \arrow[d,"h",swap] & U' \arrow[r,dashed] \arrow[d,"s"] \arrow[dl,dashed,"u",swap] & {} \\
    	U'' \arrow[r,tail,"b"] & Y \arrow[r,two heads,"d"] & I \arrow[r,dashed] & {}
    \end{tikzcd}\]
    where $s$ exists by Lemma \ref{properties of 0-Aus} (1) and $t$ by ${\rm (ET3)^{op}}$. Since $U$ is presilting, there exists $u$ such that $du=s$. Then $d(h-ua)=0$, there exists $v$ such that $h=bv+ua$. Since $f$ (resp. $g$) is right (resp. left) minimal, $a,b$ are in the radical, hence so is $h$. This completes the proof.
\end{proof}

\begin{lem}\label{interval=summand completion}
	$M\leq N$ and $\mathsf{silt}\,[M,N]=\mathsf{silt}_{U}\mathcal{C}$.
\end{lem}
\begin{proof}
	Let $T\in \mathsf{silt}\,\mathcal{C}$ such that $U\in \mathsf{add}T$. Applying the functor ${\rm Hom}_{\mathcal{C}}(-,T)$ to the $\mathbb{E}$-triangle $X\stackrel{a}\rightarrowtail U'\stackrel{f}\twoheadrightarrow I\dashrightarrow$, we obtain an exact sequence $\mathbb{E}(U',T)\rightarrow \mathbb{E}(X,T)\rightarrow \mathbb{E}^{2}(I,T)$. Since $\mathbb{E}(U',T)=0=\mathbb{E}^{2}(I,T)$, $\mathbb{E}(X,T)=0$. Thus $\mathbb{E}(N,T)=0$, $T\leq N$. In particular, $M\leq N$. Similarly applying ${\rm Hom}_{\mathcal{C}}(T,-)$ to $P\stackrel{g}\rightarrowtail U''\stackrel{b}\twoheadrightarrow Y\dashrightarrow$, we obtain $\mathbb{E}(T,M)=0$ and hence $M\leq T$.
	
	For the other direction, let $T\in \mathsf{silt}\,[M,N]$, then $U\in {^{\bot}N}\cap M^{\bot}\subseteq {^{\bot}T}\cap T^{\bot}=\mathsf{add}T$ by Theorem \ref{AT22thm5.7}.
\end{proof}

In Lemma \ref{interval=summand completion}, we abuse the notation $\mathsf{silt}\,[M,N]$ since $U$ and $M,N$ are not necessarily basic. As the subset $\mathsf{silt}_{U}\mathcal{C}$, for a $\tau$-rigid pair $(M,P)$, we define
\[s\tau\text{-}\mathsf{tilt}_{(M,P)}A:=\{(T,Q)\in s\tau\text{-}\mathsf{tilt}\,A\,|\,M\in \mathsf{add}T,\,P\in \mathsf{add}Q\}.\]

\begin{cor}
	The bijections $\mathsf{silt}\,\mathcal{C}\leftrightarrow s\tau\text{-}\mathsf{tilt}A$ restrict to bijections $\mathsf{silt}_{U}\mathcal{C}\leftrightarrow s\tau\text{-}\mathsf{tilt}_{(\overline{U},\overline{\Omega I_{U}})}A$. In particular, $N$ (resp. $M$) corresponds to Bongartz (resp. co-Bongartz) completion of the $\tau$-rigid pair $(\overline{U},\overline{\Omega I_{U}})$ in $\mathsf{mod}A$.
\end{cor}
\begin{proof}
	The restrictions follow by the construction. By \cite[Proposition 2.9]{AIR}, $s\tau\text{-}\mathsf{tilt}_{(\overline{U},\overline{\Omega I_{U}})}A$ is an interval with upper (resp. lower) bound the Bongartz (resp. co-Bongartz) completion of $(\overline{U},\overline{\Omega I_{U}})$. Thus the result follows by Corollary \ref{silt-tau-tilt} and Lemma \ref{interval=summand completion}.
\end{proof}

Next we state a lemma without assuming $\mathcal{C}$ being 0-Auslander.

\begin{lem}\label{silting bijection}
	Let $\mathcal{C}$ be an extriangulated category with enough projectives or injectives and $\mathcal{U}=\mathsf{add}\,\mathcal{U}$ be a subcategory whose objects are projective-injective. Denote by $\widetilde{(-)}:\mathcal{C}\rightarrow \widetilde{\mathcal{C}}:=\mathcal{C}/[\mathcal{U}]$ the quotient functor. Then there is a commutative diagram of bijections
	\[\begin{tikzcd}
		\mathsf{bdd\text{-}hcotors}\,\mathcal{C} \arrow[r,"\widetilde{(-)}"] \arrow[d,leftrightarrow] & \mathsf{bdd\text{-}hcotors}\,\widetilde{\mathcal{C}} \arrow[d,leftrightarrow] \\
		\mathsf{silt}\,\mathcal{C} \arrow[r,"\widetilde{(-)}"] & \mathsf{silt}\,\widetilde{\mathcal{C}}
	\end{tikzcd}\]
    where the vertical maps are given in Theorem \ref{AT22thm5.7}.
\end{lem}
\begin{proof}
	The ideal quotient $\widetilde{\mathcal{C}}$ has an induced extriangulated structure $(\widetilde{\mathcal{C}},\widetilde{\mathbb{E}},\widetilde{\mathfrak{s}})$, see \cite[Proposition 3.30]{NP}. Assume $\mathcal{C}$ has enough projectives for example. For $X,Y\in \mathcal{C}$, then $\widetilde{\mathbb{E}}(X,Y)=\mathbb{E}(X,Y)$. For $n\geq 2$, consider $\mathbb{E}$-triangles $X_{i+1}\stackrel{a_{i}}\rightarrowtail P_{i}\stackrel{b_{i}}\twoheadrightarrow X_{i}\dashrightarrow,~ 0\leq i\leq n-2$ with $X=X_{0}$ and $P_{i}$ projective. Then $\mathbb{E}^{n}(X,Y)=\mathbb{E}^{n}(X_{0},Y)\cong \mathbb{E}^{n-1}(X_{1},Y)\cong \cdots \cong \mathbb{E}(X_{n-1},Y)$. Since $X_{i+1}\stackrel{\widetilde{a_{i}}}\rightarrowtail P_{i}\stackrel{\widetilde{b_{i}}}\twoheadrightarrow X_{i}\dashrightarrow,~ 0\leq i\leq n-2$ are $\widetilde{\mathbb{E}}$-triangles, we have $\widetilde{\mathbb{E}}^{n}(X,Y)\cong \widetilde{\mathbb{E}}(X_{n-1},Y)$. Thus $\mathbb{E}^{n}(X,Y)\cong \widetilde{\mathbb{E}}^{n}(X,Y)$. Hence for a subcategory $\mathcal{T}$, $\mathcal{T}^{\bot}$ in $\mathcal{C}$ coincides with $\mathcal{T}^{\bot}$ in $\widetilde{\mathcal{C}}$.
	
	Let $(\mathcal{X},\mathcal{Y})\in \mathsf{bdd\text{-}hcotors}\,\mathcal{C}$, then there exists $\mathcal{T}\in \mathsf{silt}\,\mathcal{C}$ such that $\mathcal{X}={^{\bot}\mathcal{T}}$ and $\mathcal{Y}=\mathcal{T}^{\bot}$ in $\mathcal{C}$ by Theorem \ref{AT22thm5.7}. Thus $\mathcal{X}$ and $\mathcal{Y}$ are isomorphism closed and direct summand closed in $\widetilde{\mathcal{C}}$. It is easy to check that $(\mathcal{X},\mathcal{Y})\in \mathsf{bdd\text{-}hcotors}\,\widetilde{\mathcal{C}}$ by definition. On the other hand, assume $(\mathcal{X},\mathcal{Y})\in \mathsf{bdd\text{-}hcotors}\,\widetilde{\mathcal{C}}$, then $\mathcal{X},\mathcal{Y}$ are isomorphism closed and direct summand closed in $\mathcal{C}$ and $\mathbb{E}^{\geq 1}(\mathcal{X},\mathcal{Y})=0$. For any $T\in \widetilde{\mathcal{C}}$, there is an $\mathbb{E}$-triangle $Y\rightarrowtail X\twoheadrightarrow T\dashrightarrow$ with $Y\in \mathcal{Y}$ such that $X\cong X'$ in $\widetilde{\mathcal{C}}$ for some $X'\in \mathcal{X}$. Then $X\in \mathcal{X}$. By such fact, we obtain $(\mathcal{X},\mathcal{Y})$ is a cotorsion pair in $\mathcal{C}$ and is bounded.
	
	By Theorem \ref{AT22thm5.7}, the result follows.
\end{proof}

Clearly all bijections in Lemma \ref{silting bijection} are isomorphisms of posets.

\begin{thm}\label{reduction in 0-Aus} Let $\mathcal{C}$ be a 0-Auslander extriangulated category as before and $U$ a presilting object which is not silting. Let $\mathcal{C}':={^{\bot}M}\cap N^{\bot}$ where $N$ and $M$ are Bongartz and co-Bongartz completions of $U$ respectively. Consider the functor $\widetilde{(-)}:\mathcal{C}'\rightarrow \widetilde{\mathcal{C}'}:=\mathcal{C}'/[\mathsf{add}U]$. Then $\widetilde{(-)}$ induces an isomorphism of posets
\[{\rm red}:\mathsf{silt}_{U}\mathcal{C}\rightarrow \mathsf{silt}\,\widetilde{\mathcal{C}'}.\]
\end{thm}
\begin{proof}
	We have $\mathsf{silt}\,[M,N]=\mathsf{silt}_{U}\mathcal{C}$ by Lemma \ref{interval=summand completion}, and $\mathsf{silt}\,[M,N]\cong \mathsf{silt}\,\mathcal{C}'$ by Theorem \ref{main thm_1}. By Corollary \ref{main cor} (1) and Lemma \ref{silting bijection}, there is an isomorphism of posets $\mathsf{silt}\,\mathcal{C}'\rightarrow \mathsf{silt}\,\widetilde{\mathcal{C}'}$. Thus the result follows.
\end{proof}

\subsection{Reduction implies mutation}\label{Reduction implies mutation}

In this subsection, we recover the mutation theory in $\mathcal{C}$ (\cite[Theorem 1.3]{GNP23}) using reduction theory. All notations are the same as before.

\begin{lem}\label{C'is 0-Aus}
	$\mathcal{C}'$ is 0-Auslander and $\widetilde{\mathcal{C}'}$ is reduced 0-Auslander.
\end{lem}
\begin{proof}
	By Corollary \ref{main cor}, $\mathcal{C}'$ has enough projectives $\mathsf{add}N$, enough injectives $\mathsf{add}M$ and is hereditary (i.e. $\mathbb{E}_{\mathcal{C}'}^{2}=0$). By Lemma \ref{proj-inj}, there is an $\mathbb{E}$-triangles $P\stackrel{c}\rightarrowtail X\stackrel{a}\twoheadrightarrow U'\dashrightarrow $. Consider a commutative diagram
	\[\begin{tikzcd}
		P \arrow[r,tail,"g"] \arrow[d,tail,"c",swap] & U'' \arrow[r,two heads,"b"] \arrow[d,tail] & Y \arrow[r,dashed] \arrow[d,equal] & {} \\
		X \arrow[r,tail] \arrow[d,two heads,"a",swap] & Z \arrow[r,two heads] \arrow[d,two heads] & Y \arrow[r,dashed] & {} \\
		U' \arrow[r,equal] \arrow[d,dashed] & U' \arrow[d,dashed] & & \\
		{} & {} & &
	\end{tikzcd},\]
    then the second column splits since $U$ is presilting. By adding two $\mathbb{E}$-triangles $U\stackrel{1}\rightarrowtail U\twoheadrightarrow 0 \dashrightarrow $ and $0\rightarrowtail U\stackrel{1}\twoheadrightarrow U \dashrightarrow $ to the second row, we obtain an $\mathbb{E}$-triangle $N\rightarrowtail U_{0}\twoheadrightarrow M\dashrightarrow $. It is an $\mathbb{E}_{\mathcal{C}'}$-triangle with $U_{0}$ projective-injective. Since $\mathcal{C}'$ is Krull-Schmidt, it is easy to see for any $N'\in \mathsf{add}N$, there is an $\mathbb{E}_{\mathcal{C}'}$-triangle $N'\rightarrowtail M^{0}\twoheadrightarrow M^{1}\dashrightarrow $ with $M^{0}\in \mathsf{add}U$. Thus ${\rm dom.dim}\,\mathcal{C}\geq 1$, hence $\mathcal{C}'$ is 0-Auslander. By Lemma \ref{proj-inj}, $\widetilde{\mathcal{C}'}$ is reduced 0-Auslander.
\end{proof}

As in \cite[Corollary 3.18]{G.Jasso}, we can use our reduction theory to prove "two-completion property" for almost complete presilting objects in $\mathcal{C}$. Moreover we also obtain the $\mathbb{E}$-triangle connecting the two completions. Thus we actually recover the mutation theory in $\mathcal{C}$ as follows. In other words, reduction implies mutation.

\begin{thm}\label{another proof for GNP1.3}
	Let $\mathcal{C}$ be a 0-Auslander category as before and $U$ be a basic presilting (not silting) object. There exists objects $X'$ and $Y'$ which are not isomorphic, such that $N:=X'\oplus U$ and $M:=Y'\oplus U$ are basic silting objects and $M\leq N$. Moreover  $\mathsf{silt}\,[M,N]=\mathsf{silt}_{U}\mathcal{C}$ and there is an $\mathbb{E}$-triangle
	\[X'\stackrel{f}\rightarrowtail U_{1}\stackrel{g}\twoheadrightarrow Y'\dashrightarrow \]
	with $f$ (resp. $g$) a minimal left (resp. right) $\mathsf{add}U$-approximation. In particular, if $U$ is almost complete (i.e. $|U|=|P|-1$), then $\mathsf{silt}_{U}\mathcal{C}=\{M,N\}$ and this is just \cite[Theorem 1.3]{GNP23}.
\end{thm}
\begin{proof}
	By Lemma \ref{two completions} and \ref{interval=summand completion}, it suffices to show the existence of the $\mathbb{E}$-triangle. By Lemma \ref{C'is 0-Aus}, the subcategory $\mathcal{C}'={^{\bot}M}\cap N^{\bot}$ is 0-Auslander and $\widetilde{\mathcal{C}'}=\mathcal{C}'/[\mathsf{add}U]$ is reduced 0-Auslander. Since $X'$ and $Y'$ are basic, by Lemma \ref{properties of 0-Aus} (4), there is an $\mathbb{E}$-triangle $X'\stackrel{f}\rightarrowtail U_{1}\stackrel{g}\twoheadrightarrow Y'\dashrightarrow $ such that $U_{1}$ is zero in $\widetilde{\mathcal{C}'}$. Thus $U_{1}\in \mathsf{add}U$. Since $M$ and $N$ are basic silting, $f$ and $g$ are approximations and in the Jacobson radical. Hence they are minimal approximations. If $U$ is almost complete, let $T$ be a basic silting object in $\widetilde{\mathcal{C}'}$. There is an $\mathbb{E}_{\widetilde{\mathcal{C}'}}$-triangle $X'\stackrel{\widetilde{h}}\rightarrowtail T^{0}\twoheadrightarrow T^{1}\dashrightarrow$ with $\widetilde{h}$ left minimal and $T^{0},T^{1}\in \mathsf{add}T$ by \cite[Proposition 5.5]{AT22}. By \cite[Lemma 4.23]{GNP23}, $\mathsf{add}T^{0}\cap \mathsf{add}T^{1}=0$. Since $|T|=1$ in $\widetilde{\mathcal{C}'}$, $T^{1}=0$ or $T^{0}=0$. Thus $T=X'$ or $T=Y'$ and hence $\mathsf{silt}\,\widetilde{\mathcal{C}'}=\{X',Y'\}$. The results follow from Theorem \ref{reduction in 0-Aus}.
\end{proof}

\begin{rem}\label{rem on mutaion vs 0-Aus}
	We give some remarks on Theorem \ref{another proof for GNP1.3}.
	
	(1) Our approach is different from that of \cite[Theorem 4.26]{GNP23}. The key difference is that when reducing to a smaller category, we consider $\mathcal{C}'={^{\bot}M}\cap N^{\bot}$ and $\widetilde{\mathcal{C}'}=\mathcal{C}'/[\mathsf{add}U]$ for any presilting $U$ while they consider $\overline{\mathcal{C}}_{U}=\frac{{^{\bot_{1}}U}\cap U^{\bot_{1}}}{[\mathsf{add}U]}$ for almost complete presilting $U$. In \cite[Theorem 4.26]{GNP23}, the authors prove that $\overline{\mathcal{C}}_{U}$ is hereditary and there is a triangle connecting its projective and injective, which is close to being 0-Auslander. In fact the two subquotients are actually equal for any presilting object $U$. We give a simple proof as follows: Let $\mathcal{C}'':={^{\bot_{1}}U}\cap U^{\bot_{1}}$. Then the $\mathbb{E}$-triangle $X\rightarrowtail U'\oplus U''\twoheadrightarrow Y\dashrightarrow$ in Lemma \ref{C'is 0-Aus} implies $X$ (resp. $Y$) is projective (resp. injective) in $\mathcal{C}''$. Thus $N$ (resp. $M$) is projective (resp. injective) in $\mathcal{C}''$. We obtain $\mathcal{C}''\subseteq \mathcal{C}'$, the reverse inclusion is obvious. Hence $\widetilde{\mathcal{C}'}=\overline{\mathcal{C}}_{U}$.
	
	(2) In the proof, we see that the fact $\mathcal{C}'={^{\bot}M}\cap N^{\bot}$ is 0-Auslander is related to an $\mathbb{E}$-triangle connecting $M$ and $N$. We will discuss this feature in general in Section \ref{section 5}.
\end{rem}

\subsection{Compatibility with $\tau$-tilting reduction}
In this subsection, we investigate the relation between our reduction (Theorem \ref{reduction in 0-Aus}) in the reduced 0-Auslander extriangulated category $\mathcal{C}$ and $\tau$-tilting reduction \cite{G.Jasso} in $\mathsf{mod}A$. All notations are the same as before. Let $B:=\frac{{\rm End}_{A}(\overline{N},\overline{N})}{[\mathsf{add}\overline{U}](\overline{N},\overline{N})}$. Next we associate $\widetilde{\mathcal{C}'}$ with $\mathsf{mod}B$.

\begin{lem}\label{func iso}
	There is a functorial isomorphism ${\rm Hom}_{\widetilde{\mathcal{C}'}}(N,T)\cong \frac{{\rm Hom}_{A}(\overline{N},\overline{T})}{[\mathsf{add}\overline{U}](\overline{N},\overline{T})}$ for each $T\in \mathcal{C}'$.
\end{lem}
\begin{proof}
	By Lemma \ref{properties of 0-Aus}, we have a functorial isomorphism ${\rm Hom}_{A}(\overline{N},\overline{T})\cong \frac{{\rm Hom}_{\mathcal{C}}(N,T)}{[\mathcal{I}](N,T)}$. This isomorphism restricts to $[\mathsf{add}\overline{U}](\overline{N},\overline{T})\cong \frac{[\mathsf{add}U](N,T)+[\mathcal{I}](N,T)}{[\mathcal{I}](N,T)}$. Thus it suffices to show $[\mathcal{I}](N,T)\subseteq [\mathsf{add}U](N,T)$. Since $N=X\oplus U$, we only need to show $[\mathcal{I}](X,T)\subseteq [\mathsf{add}U](X,T)$. Consider the $\mathbb{E}$-triangle $X\stackrel{a}\rightarrowtail U'\stackrel{f}\twoheadrightarrow I\dashrightarrow$, then any morphism from $X$ to an injective object factors through $a:X\rightarrow U'$. Thus the result follows.
\end{proof}

By Lemma \ref{properties of 0-Aus}, \ref{C'is 0-Aus} and \ref{func iso}, the functor ${\rm Hom}_{\widetilde{\mathcal{C}'}}(N,-):\widetilde{\mathcal{C}'}\rightarrow \mathsf{mod}B$ induces an equivalence $\widetilde{\mathcal{C}'}/[\mathsf{add}M]\stackrel{\sim}\rightarrow \mathsf{mod}B$. By Corollary \ref{silt-tau-tilt} it also induces a bijection $\mathsf{silt}\,\widetilde{\mathcal{C}'}\rightarrow s\tau\text{-}\mathsf{tilt}B$. Next we show the map "red" in Theorem \ref{reduction in 0-Aus} is compatible with $\tau$-tilting reduction (cf. \cite{G.Jasso}).

\begin{thm}\label{main thm2}
	There is a commutative diagram
	\[\begin{tikzcd}
		\mathsf{silt}_{U}\mathcal{C} \arrow[rr,"\overline{(-)}={\rm Hom}_{\mathcal{C}}(P{,}-)"] \arrow[d,"{\rm red}"] & {} & s\tau\text{-}\mathsf{tilt}_{(\overline{U},\overline{\Omega I_{U}})}A \arrow[d,"{\rm red}"] \\
		\mathsf{silt}\,\widetilde{\mathcal{C}'} \arrow[rr,"{\rm Hom}_{\widetilde{\mathcal{C}'}}(N{,}-)"] & {} & s\tau\text{-}\mathsf{tilt}B
	\end{tikzcd}\]
	in which all maps are isomorphisms of posets. The left vertical map is given by Theorem \ref{reduction in 0-Aus} and the right one is $\tau$-tilting reduction given by \cite[Theorem 3.16]{G.Jasso}.
\end{thm}
\begin{proof}
	Let $T\in \mathsf{silt}_{U}\mathcal{C}$. Then $T\in \mathcal{C}'$ by Lemma \ref{interval=summand completion} and the left vertical map is identity on objects in the domain. By \cite[Theorem 3.16]{G.Jasso}, it suffices to show ${\rm Hom}_{A}(\overline{N},{\rm f}\overline{T})\cong {\rm Hom}_{\widetilde{\mathcal{C}'}}(N,T)$ as right $B$-module, where ${\rm f}\overline{T}$ is the torsion free part of $\overline{T}$ with respect to the torsion pair $(\mathsf{Fac}\overline{U},\overline{U}^{\bot_{0}})$. Applying ${\rm Hom}_{A}(\overline{N},-)$ to the canonical sequence $0\rightarrow {\rm t}\overline{T}\stackrel{i}\rightarrow \overline{T}\rightarrow {\rm f}\overline{T}\rightarrow 0$, we obtain an exact sequence 
	\[0\rightarrow {\rm Hom}_{A}(\overline{N},{\rm t}\overline{T})\stackrel{i_{\ast}}\rightarrow {\rm Hom}_{A}(\overline{N},\overline{T})\rightarrow {\rm Hom}_{A}(\overline{N},{\rm f}\overline{T})\rightarrow {\rm Ext}_{A}^{1}(\overline{N},{\rm t}\overline{T})=0\]
	since $\mathsf{Fac}\overline{U}\subseteq \mathsf{Fac}\overline{N}$ and $\overline{N}$ is $\tau$-rigid. By Lemma \ref{func iso}, it suffices to show ${\rm Im}(i_{\ast})=[\mathsf{add}\overline{U}](\overline{N},\overline{T})$. The inclusion $\supseteq$ is clear since every morphism $\overline{U}\rightarrow \overline{T}$ factors through $i$. For the other inclusion, consider an exact sequence $0\rightarrow K\rightarrow \overline{U_{1}}\stackrel{s}\rightarrow {\rm t}\overline{T}\rightarrow 0$ with $s$ a right $\mathsf{add}\overline{U}$-approximation. By \cite[Lemma 2.6]{AIR}, $K\in {^{\bot_{0}}(\tau \overline{U})}\cap (\overline{\Omega I_{U}})^{\bot_{0}}=\mathsf{Fac}\overline{N}$. Applying ${\rm Hom}_{A}(\overline{N},-)$, we find that any morphism $\overline{N}\rightarrow {\rm t}\overline{T}$ factors through $\overline{U_{1}}$. Thus the inclusion $\subseteq $ follows. It is easy to see the induced isomorphisms ${\rm Hom}_{A}(\overline{N},{\rm f}\overline{T})\cong \frac{{\rm Hom}_{A}(\overline{N},\overline{T})}{[\mathsf{add}\overline{U}](\overline{N},\overline{T})}\cong {\rm Hom}_{\widetilde{\mathcal{C}'}}(N,T)$ are right $B$-module isomorphisms.
\end{proof}

In Theorem \ref{main thm2}, the left vertical reduction is "simpler" than the right one.

\begin{rem}
	If we consider a $k$-linear ($k$ is a field), Hom-finite, Krull-Schmidt triangulated category $\mathcal{T}$ with a silting object $S$. Let $\mathcal{C}:=\mathsf{add}S\ast \mathsf{add}S[1]$ and $U$ be a presilting object in $\mathcal{T}$ contained in $\mathcal{C}$. Then our reduction ${\rm red}:\mathsf{silt}_{U}\mathcal{C}\rightarrow \mathsf{silt}\,\widetilde{\mathcal{C}'}$ is the same as \cite[Corollary 4.16]{G.Jasso} by Corollary \ref{generalize PZ} (2). Indeed, our "$\widetilde{\mathcal{C}'}$" is exactly the subcategory "$T_{U}\ast T_{U}[1]$" of "$\mathcal{U}$" in \cite[Proposition 4.11]{G.Jasso}. Thus Theorem \ref{main thm2} is a generalization of \cite[Theorem 4.12]{G.Jasso} but a bit simpler on the left vertical map in the commutative square, so that we don't need the commutativity to prove the bijectivity of it.
\end{rem}

We state another version of Theorem \ref{main thm2} and show the compatibility between them. Recall that $N$ (resp. $M$) is the  Bongartz (resp. co-Bongartz) completion of the presilting object $U$. 
 Let $m$ (resp. $n$) be the cotorsion pair $(\mathcal{X}_{1},\mathcal{Y}_{1})$ (resp. $(\mathcal{X}_{2},\mathcal{Y}_{2})$) corresponding to $M$ (resp. $N$) and $\overline{m}$ (resp. $\overline{n}$) be the left weak cotorsion-torsion triple $(\overline{\mathcal{X}_{1}},\overline{\mathcal{Y}_{1}},\overline{\mathcal{Y}_{1}}^{\bot_{0}})$ (resp. $(\overline{\mathcal{X}_{2}},\overline{\mathcal{Y}_{2}},\overline{\mathcal{Y}_{2}}^{\bot_{0}})$).

\begin{cor}
	There is a commutative diagram
	\[\begin{tikzcd}
		& \mathsf{cotors}\,[m,n] \arrow[rr] \arrow[dd] \arrow[dl,"{\rm red'}"] & & \mathsf{lw\text{-}\mathsf{cotors}\text{-}\mathsf{tors}}[\overline{m},\overline{n}] \arrow[dl,"{\rm red'}"] \arrow[dd] \\
		\mathsf{cotors}\,\widetilde{\mathcal{C}'} \arrow[rr] \arrow[dd] & & \mathsf{lw\text{-}\mathsf{cotors}\text{-}\mathsf{tors}}B \arrow[dd] & \\
		& \mathsf{silt}_{U}\mathcal{C} \arrow[dl,"{\rm red}"] \arrow[rr] & & s\tau\text{-}\mathsf{tilt}_{(\overline{U},\overline{\Omega I_{U}})}A \arrow[dl,"{\rm red}"] \\
		\mathsf{silt}\,\widetilde{\mathcal{C}'} \arrow[rr] & & s\tau\text{-}\mathsf{tilt}B &
	\end{tikzcd}\]
    in which all maps are isomorphisms of posets. The top left reduction is given by Corollary \ref{main cor} (1) and Lemma \ref{silting bijection}. The top right reduction is given by \cite[Theorem 3.14]{G.Jasso} (see Remark \ref{uniquely determined}).
\end{cor}
\begin{proof}
	The bottom square is the diagram in Theorem \ref{main thm2}. The left square is commutative by the proof of Theorem \ref{main thm_1} and Lemma \ref{silting bijection}. The front and back squares are commutative by Remark \ref{commutative square}. The right square is commutative by the proof of \cite[Theorem 3.16]{G.Jasso}. Since all maps are isomorphisms of posets, the top square is also commutative. But we  give a direct proof here. Let $(\mathcal{X},\mathcal{Y})\in \mathsf{cotors}\,[m,n]$, it suffices to show the two images of $\mathcal{Y}$ coincide. That is, ${\rm Hom}_{\widetilde{\mathcal{C}'}}(N,\mathcal{X}_{1}\cap \mathcal{Y})={\rm Hom}_{A}(\overline{N},\overline{\mathcal{Y}}\cap {\overline{U}}^{\bot_{0}})$. For each $T\in \mathcal{X}_{1}\cap \mathcal{Y}$, by the proof of Theorem \ref{main thm2}, we have ${\rm Hom}_{\widetilde{\mathcal{C}'}}(N,T)\cong {\rm Hom}_{A}(\overline{N},{\rm f}\overline{T})$ where ${\rm f}\overline{T}\in \overline{\mathcal{Y}}\cap {\overline{U}}^{\bot_{0}}$. Let $Z\in \mathcal{Y}$ such that $\overline{Z}\in \overline{\mathcal{Y}}\cap {\overline{U}}^{\bot_{0}}$. Then ${\rm Hom}_{A}(\overline{N},\overline{Z})\in \mathsf{mod}B$. Thus there is $T\in \widetilde{\mathcal{C}'}$ such that ${\rm Hom}_{\widetilde{\mathcal{C}'}}(N,T)\cong {\rm Hom}_{A}(\overline{N},\overline{Z})$, then ${\rm Hom}_{A}(\overline{N},{\rm f}\overline{T})\cong {\rm Hom}_{A}(\overline{N},\overline{Z})$. Since $T,Z\in \mathcal{Y}_{2}$, we have ${\rm f}\overline{T},\overline{Z}\in \overline{\mathcal{Y}_{2}}=\mathsf{Fac}\overline{N}$. Hence ${\rm f}\overline{T}\cong \overline{Z}$. In the exact sequence $0\rightarrow {\rm t}\overline{T}\rightarrow \overline{T}\rightarrow {\rm f}\overline{T}\rightarrow 0$, ${\rm t}\overline{T}\in \mathsf{Fac}\overline{U}=\overline{\mathcal{Y}_{1}}\subseteq \overline{\mathcal{Y}}$ and ${\rm f}\overline{T}\cong \overline{Z}\in \overline{\mathcal{Y}}$, hence $\overline{T}\in \overline{\mathcal{Y}}$. Thus $T\in \mathcal{Y}\cap \mathcal{X}_{1}$. We complete the proof.
\end{proof}

\section{Mutation and 0-Auslander extriangulated categories}\label{section 5}

Let $\mathcal{C}=(\mathcal{C},\mathbb{E},\mathfrak{s})$ be an extriangulated category. A natural question is when the subcategory $\mathcal{C}'={^{\bot}\mathcal{M}}\cap \mathcal{N}^{\bot}$ in Corollary \ref{main cor} (2) has good properties, for example, being 0-Auslander. Lemma \ref{C'is 0-Aus} gives us such an example $\mathcal{C}$ and in Remark \ref{rem on mutaion vs 0-Aus} we pointed out a feature of the example. In this section, we discuss it in a general setting. First we recall silting mutation in extriangulated categories introduced in \cite{AT23}, which is a direct generalization of that in triangulated categories \cite{AI}.

\begin{defn}\label{silt mutation in extri}
	Let $\mathcal{C}=(\mathcal{C},\mathbb{E},\mathfrak{s})$ be an extriangulated category.
	
	(1) \cite[Definition 4.7]{AT23} Let $\mathcal{N}$ be a subcategory of $\mathcal{C}$, and $\mathcal{D}$ be a covariantly (resp. contravariantly) finite subcategory of $\mathcal{N}$. $\mathcal{D}$ is {\em good} if each $N\in \mathcal{N}$ has a left (resp. right) $\mathcal{D}$-approximation which is an $\mathfrak{s}$-inflation (resp. $\mathfrak{s}$-deflation).
	
	(2) \cite[Definition 4.10]{AT23} Let $\mathcal{N}$ be a presilting subcategory of $\mathcal{C}$ and $\mathcal{D}$ be a good covariantly finite subcategory of $\mathcal{N}$. For each $N\in \mathcal{N}$ take an $\mathbb{E}$-triangle
	\[N\stackrel{f}\rightarrowtail D\twoheadrightarrow M_{N}\dashrightarrow \]
	with $f$ a left $\mathcal{D}$-approximation. Define a subcategory $\mu^{L}(\mathcal{N};\mathcal{D})$ of $\mathcal{C}$ to be
	\[\mu^{L}(\mathcal{N};\mathcal{D}):=\mathsf{add}(\mathcal{D}\cup \{M_{N}\,|\,N\in \mathcal{N}\}),\]
	and call it a {\em left mutation} of $\mathcal{N}$ with respect to $\mathcal{D}$. Dually we define a {\em right mutation} $\mu^{R}(\mathcal{N};\mathcal{D})$ of $\mathcal{N}$ with respect to $\mathcal{D}$.
\end{defn}

\begin{rem}
    (1) In some literature (e.g. \cite{ZZ}, \cite{FMP}), they use "strong" instead of "good" in Definition \ref{silt mutation in extri} (1).

    (2) By \cite[Remark 4.11]{AT23}, left and right mutations are well-defined (that is, independent of the choices of approximations). Indeed, if there are two $\mathbb{E}$-triangles $N\stackrel{f}\rightarrowtail D\twoheadrightarrow M_{N}\dashrightarrow$ and $N\stackrel{f'}\rightarrowtail D'\twoheadrightarrow M_{N}'\dashrightarrow$ such that $f,f'$ are left $\mathcal{D}$-approximations, then we have $M_{N}\oplus D'\cong M_{N}'\oplus D$. Moreover $\mathcal{D}$ can be assumed that $\mathcal{D}=\mathsf{add}\mathcal{D}$.
    
    (3) The silting objects $M$ and $N$ in Theorem \ref{another proof for GNP1.3} are mutations of each other in the sense of Definition \ref{silt mutation in extri}.
\end{rem}

\begin{lem}\cite[Theorem 4.12, Proposition 4.13]{AT23}\label{basic lem on mutation}
	Let $\mathcal{N}$ be a presilting subcategory of $\mathcal{C}$ and $\mathcal{D}$ be a good covariantly finite subcategory of $\mathcal{N}$. Define $\mathcal{M}:=\mu^{L}(\mathcal{N};\mathcal{D})$.
	
	(1) $\mathcal{M}$ is a presilting subcategory of $\mathcal{C}$ and $\mathbb{E}^{\geq 1}(\mathcal{N},\mathcal{M})=0$.
	
	(2) $\mathcal{M}\cap \mathcal{N}=\mathsf{add}\mathcal{D}$, and $\mathcal{M}=\mathcal{N}$ if and only if $\mathcal{N}=\mathsf{add}\mathcal{D}$.
	
	(3) If $\mathcal{N}$ is silting then so is $\mathcal{M}$.
	
	(4) $\mathcal{D}$ is a good contravariantly finite subcategory of $\mathcal{M}$ and $\mu^{R}(\mathcal{M};\mathcal{D})=\mathcal{N}$.
\end{lem}
\begin{proof}
	We only prove the first half of (2), the rest can be found in \cite{AT23}. For each $X\in \mathcal{M}\cap \mathcal{N}$, take an $\mathbb{E}$-triangle $X\stackrel{f}\rightarrowtail D\twoheadrightarrow Y\stackrel{\delta}\dashrightarrow $ with $D\in \mathcal{D}$ and $f$ a left $\mathcal{D}$-approximation. Then $Y\in \mathcal{M}$. Since $\mathcal{M}$ is presilting by (1), $\delta=0$ and hence $X\in \mathsf{add}\mathcal{D}$.
\end{proof}

\begin{thm}\label{main thm: new 0-Aus}
	Let $\mathcal{M},\mathcal{N}\in \mathsf{silt}\,\mathcal{C}$ and $\mathcal{C}':={^{\bot}\mathcal{M}}\cap \mathcal{N}^{\bot}$. Then $\mathcal{M}\leq \mathcal{N}$ and $\mathcal{C}'$ is 0-Auslander if and only if $\mathcal{M}=\mu^{L}(\mathcal{N};\mathcal{D})$ for some good covariantly finite subcategory $\mathcal{D}$ of $\mathcal{N}$.
\end{thm}
\begin{proof}
	Assume $\mathcal{M}\leq \mathcal{N}$ and $\mathcal{C}'$ is 0-Auslander. By Corollary \ref{main cor}, $\mathcal{C}'$ has enough projectives $\mathcal{N}$ and enough injectives $\mathcal{M}$. Let $\mathcal{D}:=\mathcal{M}\cap \mathcal{N}$. Then by definition for each $N\in \mathcal{N}$ there is an $\mathbb{E}$-triangle $N\stackrel{f}\rightarrowtail D\twoheadrightarrow M_{N}\dashrightarrow $ with $D\in \mathcal{D},M_{N}\in \mathcal{M}$. Since $\mathcal{M}$ is silting, $f$ is a left $\mathcal{D}$-approximation. Hence $\mathcal{D}$ is good covariantly finite in $\mathcal{N}$ and $\mu^{L}(\mathcal{N};\mathcal{D})\subseteq \mathcal{M}$. For each $M\in \mathcal{M}$, there is an $\mathbb{E}$-triangle $N_{M}\stackrel{f'}\rightarrowtail D'\twoheadrightarrow M\dashrightarrow $ with $D'\in \mathcal{D},N_{M}\in \mathcal{N}$ and $f'$ a left $\mathcal{D}$-approximation. Thus $\mathcal{M}=\mu^{L}(\mathcal{N};\mathcal{D})$.
	
	On the other hand, assume $\mathcal{M}=\mu^{L}(\mathcal{N};\mathcal{D})$ for some $\mathcal{D}$. Then $\mathcal{M}\leq \mathcal{N}$ by Lemma \ref{basic lem on mutation}. We claim that $\mathcal{C}'$ is hereditary. By \cite[Lemma 4.9(2)]{AT23}, we have $\mathbb{E}^{\geq 2}(\mathcal{M},\mathcal{N})=0$. For each $N_{1}\in \mathcal{N}_{1}^{\wedge}$, there is an $\mathbb{E}$-triangle $N'\rightarrowtail N\twoheadrightarrow N_{1}\dashrightarrow $ with $N',N\in \mathcal{N}$. Applying ${\rm Hom}_{\mathcal{C}}(\mathcal{M},-)$, we obtain $\mathbb{E}^{\geq 2}(\mathcal{M},\mathcal{N}_{1}^{\wedge})=0$. By similar arguments, $\mathbb{E}^{\geq 2}(\mathcal{M},\mathcal{N}^{\wedge})=0$. Dually, we obtain $\mathbb{E}^{\geq 2}(\mathcal{M}^{\vee},\mathcal{N}^{\wedge})=0$. Since $\mathcal{M}^{\vee}={^{\bot}\mathcal{M}}$ and $\mathcal{N}^{\wedge}=\mathcal{N}^{\bot}$ by Theorem \ref{AT22thm5.7}, we have $\mathbb{E}^{\geq 2}(\mathcal{C}',\mathcal{C}')=0$. By Corollary \ref{main cor} (2), $\mathbb{E}_{\mathcal{C}'}^{\geq 2}=0$. Thus $\mathcal{C}'$ is hereditary. For each $N\in \mathcal{N}$, there is an $\mathbb{E}$-triangle $N\stackrel{f}\rightarrowtail D\twoheadrightarrow M_{N}\dashrightarrow $ with $D\in \mathcal{D},M_{N}\in \mathcal{M}$ and $f$ a left $\mathcal{D}$-approximation. Since $\mathcal{D}\subseteq \mathcal{M}\cap \mathcal{N}$, by Corollary \ref{main cor} (1), this is an $\mathbb{E}_{\mathcal{C}'}$-triangle with $D$ projective-injective. Thus ${\rm dom.dim}\,\mathcal{C}\geq 1$ and hence $\mathcal{C}'$ is 0-Auslander.
\end{proof}

Theorem \ref{main thm: new 0-Aus} provides a new construction of 0-Auslander extriangulated categories, that is, 0-Auslander extension closed subcategory $\mathcal{C}'$ arising from silting mutation in a larger extriangulated category $\mathcal{C}$. Note that any 0-Auslander extriangulated category $\mathcal{C}$ arises in this way: its projectives $\mathcal{P}$ and injectives $\mathcal{I}$ are silting subcategories. We have $\mathcal{I}=\mu^{L}(\mathcal{P};\mathcal{P}\cap \mathcal{I})$ and $\mathcal{C}={^{\bot}\mathcal{I}}\cap \mathcal{P}^{\bot}$. 

Theorem \ref{main thm: new 0-Aus} can be regarded as an alternative way to define silting mutation.

\begin{exam}\label{examples} Some well-known examples of 0-Auslander extriangulated categories can also arise from silting mutation. 
	
	(1) Let $\mathcal{T}$ be a triangulated category and $\mathcal{S}$ be a silting subcategory. The triangle $S\rightarrow 0\rightarrow S[1]\rightarrow S[1]$ connects silting subcategories $\mathcal{S}$ and $\mathcal{S}[1]=\mu^{L}(\mathcal{S};0)$ for each $S\in \mathcal{S}$. Thus the two-term category $\mathcal{S}\ast \mathcal{S}[1]={^{\bot}\mathcal{S}[1]}\cap \mathcal{S}^{\bot}$ is 0-Auslander.
	
	(2) The 0-Auslander subcategory $\mathcal{C}'$ in Lemma \ref{C'is 0-Aus} also arises from silting mutation.
	
	(3) Assume $\mathcal{C}$ has enough injectives $\mathcal{I}$. Let $\mathcal{R}=\mathsf{add}\mathcal{R}$ be a rigid subcategory of $\mathcal{C}$ containing $\mathcal{I}$. Consider the relative substructure $\mathbb{E}_{\mathcal{R}}(C,A)=\{\delta\in \mathbb{E}(C,A)\,|\,\forall f:R\rightarrow C\text{ with }R\in \mathcal{R},f^{\ast}\delta=0\}$. Then $(\mathcal{C},\mathbb{E}_{\mathcal{R}},\mathfrak{s}|_{\mathbb{E}_{\mathcal{R}}})$ is extriangulated and all objects in $\mathcal{R}$ are projective (see \cite{HLN}, \cite{GNP23}). Under this substructure, consider $\mathcal{C}'=\mathsf{thick}\mathcal{R}=\mathcal{R}^{\wedge}$ (see \cite[Proposition 4.10]{AT22}), clearly positive extensions in $\mathcal{C}'$ coincide with those in $\mathcal{C}$. Then $\mathcal{R}\in \mathsf{silt}\,\mathcal{C}'$. For each $N\in \mathcal{R}$, there is an $\mathbb{E}$-triangle $N\rightarrowtail I\twoheadrightarrow M\dashrightarrow$ with $I\in \mathcal{I}$. It is an $\mathbb{E}_{\mathcal{R}}$-triangle with $M\in \mathcal{C}'$. Then in $\mathcal{C}'$, $\mu^{L}(\mathcal{R};\mathcal{I})$ exists. By Theorem \ref{main thm: new 0-Aus}, under the substructure, $\mathcal{D}:={^{\bot}\mu^{L}(\mathcal{R};\mathcal{I})}\cap \mathcal{R}^{\bot}\cap \mathcal{C}'={^{\bot_{\mathcal{C}'}}\mu^{L}(\mathcal{R};\mathcal{I})}\cap \mathcal{R}^{\bot_{\mathcal{C}'}}$ is 0-Auslander. We claim $\mathcal{D}={\rm Cone}_{(\mathcal{C},\mathbb{E},\mathfrak{s})}(\mathcal{R},\mathcal{R})$. Clearly the inclusion $\subseteq$ holds. For the other direction, it suffices to show for each $X\in {\rm Cone}_{(\mathcal{C},\mathbb{E},\mathfrak{s})}(\mathcal{R},\mathcal{R})$, $\mathbb{E}_{\mathcal{R}}^{\geq 1}(X,M)=0$. Since $\mathbb{E}_{\mathcal{R}}^{\geq 2}(X,-)=0$ and $I$ is injective under the substructure, the result follows. In particular, the two-term subcategory $\mathcal{R}\ast \mathcal{R}[1]$ of a triangulated category is a 0-Auslander extriangulated category for a rigid subcategory $\mathcal{R}.$
	
	This example is mentioned in 3.3.10 of \cite{GNP23} and it unifies many well-known examples (see \cite{GNP23}).
\end{exam}

\begin{lem}\label{relation between two mutations}
	Let $\mathcal{N}\in \mathsf{silt}\,\mathcal{C}$. If $\mathcal{D}_{1}\subseteq \mathcal{D}_{2}\subseteq \mathcal{N}$ such that $\mathcal{D}_{1},\mathcal{D}_{2}$ are good covariantly finite subcategories of $\mathcal{N}$. Then $\mathcal{N}\geq \mu^{L}(\mathcal{N};\mathcal{D}_{2})\geq \mu^{L}(\mathcal{N};\mathcal{D}_{1})$. Moreover we have $\mathsf{silt}\,[\mu^{L}(\mathcal{N};\mathcal{D}_{2}),\mathcal{N}]=\mathsf{silt}_{\mathcal{D}_{2}}[\mu^{L}(\mathcal{N};\mathcal{D}_{1}),\mathcal{N}]:=\{\mathcal{T}\in \mathsf{silt}\,\mathcal{C}\,|\,\mu^{L}(\mathcal{N};\mathcal{D}_{1})\leq \mathcal{T} \leq \mathcal{N},~\mathcal{D}_{2}\subseteq \mathcal{T}\}$.
\end{lem}
\begin{proof}
	For any $N,N'\in \mathcal{N}$, take $\mathbb{E}$-triangles $N\stackrel{f}\rightarrowtail D_{2}\stackrel{a}\twoheadrightarrow M\stackrel{\delta_{2}}\dashrightarrow $ and $N'\stackrel{g}\rightarrowtail D_{1}\stackrel{b}\twoheadrightarrow M'\stackrel{\delta_{1}}\dashrightarrow $ where $D_{1}\in \mathcal{D}_{1}, D_{2}\in \mathcal{D}_{2}$, $f$ is a left $\mathcal{D}_{2}$-approximation and $g$ is a left $\mathcal{D}_{1}$-approximation. Since $\mathbb{E}^{\geq 1}(\mathcal{D}_{2},\mathcal{D}_{1})=0$ and $\mathbb{E}^{\geq 1}(\mathcal{D}_{2},M')=0$, it suffices to show $\mathbb{E}^{\geq 1}(M,\mathcal{D}_{1})=0$ and $\mathbb{E}^{\geq 1}(M,M')=0$. Since $f$ is an approximation and $\mathcal{D}_{1}\subseteq \mathcal{D}_{2}$, by applying ${\rm Hom}_{\mathcal{C}}(-,\mathcal{D}_{1})$ to the first $\mathbb{E}$-triangle, we obtain $\mathbb{E}^{\geq 1}(M,\mathcal{D}_{1})=0$. For any morphism $h:N\rightarrow M'$, consider the following commutative diagram
	\[\begin{tikzcd}
		& N \arrow[r,tail,"f"] \arrow[dr,"h",near start] \arrow[d,dashed,"c",swap] & D_{2} \arrow[r,two heads,"a"] \arrow[dl,dashed,"d",near start] & M \arrow[r,dashed] & {} \\
		N' \arrow[r,tail,"g",swap] & D_{1} \arrow[r,two heads,"b",swap] & M' \arrow[r,dashed] & {} & .
	\end{tikzcd}\]
	Since $\mathcal{N}$ is silting, there is $c:N\rightarrow D_{1}$ such that $h=bc$. Since $\mathcal{D}_{1}\subseteq \mathcal{D}_{2}$ and $f$ is an approximation, there is $d:D_{2}\rightarrow D_{1}$ such that $c=df$. Thus $h$ factors through $f$. Applying ${\rm Hom}_{\mathcal{C}}(-,M')$ to the first row, we obtain $\mathbb{E}^{\geq 1}(M,M')=0$.
	
	For any silting $\mathcal{T}$ such that $\mu^{L}(\mathcal{N};\mathcal{D}_{2})\leq \mathcal{T}\leq \mathcal{N}$, as in Lemma \ref{interval=summand completion}, we have $\mathcal{D}_{2}\subseteq {^{\bot}\mathcal{N}}\cap \mu^{L}(\mathcal{N};\mathcal{D}_{2})^{\bot} \subseteq {^{\bot}\mathcal{T}}\cap \mathcal{T}^{\bot}=\mathcal{T}$. For the other inclusion, let $\mathcal{T}\in \mathsf{silt}_{\mathcal{D}_{2}}[\mu^{L}(\mathcal{N};\mathcal{D}_{1}),\mathcal{N}]$. For each $N\in \mathcal{N}$, there is a commutative diagram
	\[\begin{tikzcd}
		N \arrow[r,tail,"f"] \arrow[d,tail,"g",swap] & D_{2} \arrow[r,two heads] \arrow[d,tail] & N_{2} \arrow[r,dashed] \arrow[d,equal] & {} \\
		D_{1} \arrow[d,two heads] \arrow[r,tail] & M \arrow[r,two heads] \arrow[d,two heads] & N_{2} \arrow[r,dashed] & {} \\
		N_{1} \arrow[d,dashed] \arrow[r,equal] & N_{1} \arrow[d,dashed] & & \\
		{} & {} & &
	\end{tikzcd}\]
    with $D_{1}\in \mathcal{D}_{1}$ (resp. $D_{2}\in \mathcal{D}_{2}$) and $g$ (resp. $f$) a left $\mathcal{D}_{1}$ (resp. $\mathcal{D}_{2}$)-approximation. Then $\mathbb{E}(N_{2},D_{1})=0$, hence $M\cong D_{1}\oplus N_{2}$. Applying ${\rm Hom}_{\mathcal{C}}(\mathcal{T},-)$ to the middle column, we obtain $\mathbb{E}^{\geq 1}(\mathcal{T},N_{2})=0$. Thus $\mathcal{T}\geq \mu^{L}(\mathcal{N};\mathcal{D}_{2})$.
\end{proof}

\begin{cor}
	Let $\mathcal{N}\in \mathsf{silt}\,\mathcal{C}$. If there is a sequence $\mathcal{N}=\mathcal{N}_{0}\geq \mathcal{N}_{1}\geq \cdots \geq \mathcal{N}_{r},r\geq 2$ such that $\mathcal{N}_{i}=\mu^{L}(\mathcal{N}_{i-1};\mathcal{D}),1\leq i\leq r$ for a subcategory $\mathcal{D}=\mathsf{add}\mathcal{D}\subsetneq \mathcal{N}$. Then there is no subcategory $\mathcal{D}'$ such that $\mathcal{N}_{r}\geq \mu^{L}(\mathcal{N};\mathcal{D}')$.
\end{cor}
\begin{proof}
	Suppose there is a $\mathcal{D}'\subseteq \mathcal{N}$ such that $\mathcal{N}_{r}\geq \mu^{L}(\mathcal{N};\mathcal{D}')$. For any silting $\mathcal{T}$ such that $\mu^{L}(\mathcal{N};\mathcal{D}')\leq \mathcal{T}\leq \mathcal{N}$, as usual, we have $\mathcal{D}'\subseteq \mathcal{T}$. Then $\mathcal{D}'\subseteq \mathcal{N}_{i}$ for all $i$ and hence $\mathcal{D}'\subseteq \mathcal{D}$ by Lemma \ref{basic lem on mutation}(2). By Lemma \ref{relation between two mutations}, $\mathcal{N}_{i}$ lies in the interval $\mathsf{silt}[\mathcal{N}_{1},\mathcal{N}]$. Thus $\mathcal{N}_{1}\geq \cdots \geq \mathcal{N}_{r}\geq \mathcal{N}_{1}$. Then $\mathcal{D}=\mathcal{N}$ by Lemma \ref{basic lem on mutation}(2), a contradiction.
\end{proof}

\begin{lem}\label{mutation invariant}
	Assume $\mathcal{C}$ has enough projectives or injectives. Let $\mathcal{U}=\mathsf{add}\,\mathcal{U}$ be a subcategory whose objects are projective-injective, $\mathcal{M},\mathcal{N}\in \mathsf{silt}\,\mathcal{C}$ such that $\mathcal{M}\leq \mathcal{N}$ and $\mathcal{D}:=\mathcal{M}\cap \mathcal{N}$. Then $\mathcal{M}=\mu^{L}(\mathcal{N};\mathcal{D})$ in $\mathcal{C}$ if and only if $\mathcal{M}=\mu^{L}(\mathcal{N};\mathcal{D})$ in $\widetilde{\mathcal{C}}:=\mathcal{C}/[\mathcal{U}]$. In particular, if $\mathcal{C}$ is Krull-Schmidt, $\mathcal{M}$ is an irreducible left mutation (see the definition below) of $\mathcal{N}$ in $\mathcal{C}$ if and only if it is in $\widetilde{\mathcal{C}}$.
\end{lem}
\begin{proof}
	Since $\mathcal{U}\subseteq \mathcal{D}$ and $\mathbb{E}(X,Y)=\widetilde{\mathbb{E}}(X,Y)$ for any $X,Y$, the results follow immediately by definition and Lemma \ref{silting bijection}.
\end{proof}

We give some simple applications of Theorem \ref{main thm: new 0-Aus} in the rest of this section.

If $\mathcal{C}$ is Krull-Schmidt. Let $\mathcal{N}$ be a silting subcategory and $\mathcal{X}=\mathsf{add}\mathcal{X}$ be a subcategory of $\mathcal{N}$. Define $\mathcal{N}_{\mathcal{X}}:=\mathsf{add}(\mathsf{ind}\mathcal{N}\setminus \mathsf{ind}\mathcal{X})$ and $\mu^{L}_{\mathcal{X}}(\mathcal{N}):=\mu^{L}(\mathcal{N};\mathcal{N}_{\mathcal{X}})$ (if $\mathcal{N}_{\mathcal{X}}$ is a good covariantly finite subcategory of $\mathcal{N}$). If $\mathcal{X}=\mathsf{add}X$, then we denote by $\mathcal{N}_{X}$ and $\mu_{X}^{L}(\mathcal{N})$ instead. If $X$ is indecomposable, we call it an {\em irreducible} left mutation. The following is an immediate corollary. 

\begin{cor}\cite[Theorem 4.16]{AT23}
	Assume $\mathcal{C}$ is Krull-Schmidt. Let $\mathcal{N}\in \mathsf{slit}\,\mathcal{C}$ and $\mathcal{M}$ is an irreducible left mutation of $\mathcal{N}$. Then $\mathsf{silt}\,[\mathcal{M},\mathcal{N}]=\{\mathcal{M},\mathcal{N}\}$.
\end{cor}
\begin{proof}
	Let $\mathcal{C}'={^{\bot}\mathcal{M}}\cap \mathcal{N}^{\bot}$ and $\widetilde{\mathcal{C}'}=\mathcal{C}'/[\mathcal{M}\cap \mathcal{N}]$. Then $\mathcal{C}'$ is 0-Auslander and $\widetilde{\mathcal{C}'}$ is reduced 0-Auslander by Theorem \ref{main thm: new 0-Aus}. By Theorem \ref{main thm_1} and Lemma \ref{silting bijection},  $\mathsf{silt}\,[\mathcal{M},\mathcal{N}]\cong \mathsf{silt}\,\mathcal{C}'\cong \mathsf{silt}\,\widetilde{\mathcal{C}'}$. Thus the result follows from the same proof as that of Theorem \ref{another proof for GNP1.3}.
\end{proof}

Assume for the rest of this section that $\mathcal{C}$ is $k$-linear ($k$ is a field), Hom-finite, Krull-Schmidt. The next result is a generalization of \cite[Proposition 3.3]{AIR}, which concerns two-term silting complex.

\begin{cor}
Assume $\mathcal{C}$ admits a silting object $N$ and $X\in \mathsf{add}N$ satisfies that $M:=\mu^{L}_{X}(N)$ exists. Let $\mathcal{T}$ be a presilting subcategory such that $N\geq \mathcal{T}\geq M$ (i.e. $\mathcal{T}\subseteq \mathcal{C}':={^{\bot}M}\cap N^{\bot}$). Then $|\mathcal{T}|\leq |N|$ and $|\mathcal{T}|=|N|$ if and only if $\mathcal{T}$ is silting.
\end{cor}
\begin{proof}
	By Corollary \ref{main cor}(2), $\mathcal{T}$ is a presilting subcategory of $\mathcal{C}'$. Let $\widetilde{\mathcal{C}'}=\mathcal{C}'/[(\mathsf{add}N)_{X}]$, then by Theorem \ref{main thm: new 0-Aus} and Lemma \ref{basic lem on mutation}, it is reduced 0-Auslander. Consider $\mathcal{T}$ in the subquotient $\widetilde{\mathcal{C}'}$, denote by $|\mathcal{T}|_{\widetilde{\mathcal{C}'}}$ the number of indecomposable objects in $\mathcal{T}$. By Corollary \ref{silt-tau-tilt}, $|\mathcal{T}|_{\widetilde{\mathcal{C}'}}\leq |X|$ and the equality holds if and only if $\mathcal{T}$ is silting in $\widetilde{\mathcal{C}'}$. Hence the results follow from Lemma \ref{silting bijection} and Theorem \ref{main thm_1}.
\end{proof}

\begin{cor}\label{silting between mutation}
	Let $\mathcal{N}$ be a silting subcategory of $\mathcal{C}$ and let $X,Y\in \mathcal{N}$ such that $\mu^{L}_{X}(\mathcal{N})$ exists, $Y\in \mathsf{add}X$. Then the following statements hold.
	
	(1) $\mu^{L}_{Y}(\mathcal{N})$ exists.
	
	(2) $\mathsf{silt}\,[\mu^{L}_{X}(\mathcal{N}),\mathcal{N}]$ is a finite set if and only if $A=\frac{{\rm End}_{\mathcal{C}}(X)}{[\mathcal{N}_{X}](X,X)}$ is $\tau$-tilting finite (see \cite[Definition 1.1]{DIJ19}).
	
	(3) The Hasse quiver of $\mathsf{silt}\,[\mu^{L}_{X}(\mathcal{N}),\mathcal{N}]$ is $|X|$-regular (i.e. for each vertex, there are $|X|$ arrows attached to it) and $\mathsf{silt}\,[\mu^{L}_{Y}(\mathcal{N}),\mathcal{N}]=\mathsf{silt}_{\mathcal{N}_{Y}}[\mu^{L}_{X}(\mathcal{N}),\mathcal{N}]$ is a $|Y|$-regular subquiver.
\end{cor}
\begin{proof}
	(1) We can prove it directly using \cite[Corollary 3.16]{NP} but we give another proof here. Let $\mathcal{C}'={^{\bot}\mu^{L}_{X}(\mathcal{N})}\cap \mathcal{N}^{\bot}$ and $\widetilde{\mathcal{C}'}=\mathcal{C}'/[\mathcal{N}_{X}]$. Then $\widetilde{\mathcal{C}'}$ is reduced 0-Auslander. Since $|\mathcal{N}_{Y}|_{\widetilde{\mathcal{C}'}}$ is finite, it is a presilting object. By Theorem \ref{another proof for GNP1.3}, there are Bongartz and co-Bongartz completions and they are mutations of each other. By Lemma \ref{mutation invariant}, the result follows.
	
	(2) By Lemma \ref{properties of 0-Aus} and Corollary \ref{silt-tau-tilt}, there is a functor $\widetilde{\mathcal{C}'}\rightarrow \mathsf{mod}A$ and it induces bijections $\mathsf{silt}\,\widetilde{\mathcal{C}'}\leftrightarrow s\tau\text{-}\mathsf{tilt}\,A$. Since $\mathsf{silt}\,[\mu^{L}_{X}(\mathcal{N}),\mathcal{N}]\cong \mathsf{silt}\,\mathcal{C}'\cong \mathsf{silt}\,\widetilde{\mathcal{C}'}$ by Theorem \ref{main thm_1} and Lemma \ref{silting bijection}, the result follows.
	
	(3) Since the bijections in (2) are all isomorphisms of posets and $|A|=|X|$, the first half follows. The other half follows from Lemma \ref{relation between two mutations}.
\end{proof}

For a finite dimensional $k$-algebra $A$, we recall the definition of a maximal green sequence.

\begin{defn}\cite[Definition 4.8]{BST19}
	$A$ has a {\em maximal green sequence} if there is a finite sequence of torsion classes $0=\mathcal{T}_{0}\subsetneq \mathcal{T}_{1}\subsetneq \cdots \subsetneq \mathcal{T}_{r-1}\subsetneq \mathcal{T}_{r}=\mathsf{mod}A$ such that for all $i\in \{1,\cdots,r\}$, $\mathcal{T}_{i}$ is a {\em cover} of $\mathcal{T}_{i-1}$ (i.e. if a torsion class $\mathcal{T}$ satisfies $\mathcal{T}_{i-1}\subseteq \mathcal{T}\subseteq \mathcal{T}_{i}$, then $\mathcal{T}=\mathcal{T}_{i-1}$ or $\mathcal{T}=\mathcal{T}_{i}$).
\end{defn}

\begin{cor}\label{mutation vs maximal green sequence}
	Let $\mathcal{N}$ be a silting subcategory of $\mathcal{C}$ and let $X\in \mathcal{N}$ such that $\mu^{L}_{X}(\mathcal{N})$ exists. The following are equivalent.
	
	(1) There is a finite sequence of irreducible left mutations from $\mathcal{N}$ to $\mu^{L}_{X}(\mathcal{N})$.
	
	(2) The algebra $A=\frac{{\rm End}_{\mathcal{C}}(X)}{[\mathcal{N}_{X}](X,X)}$ has a maximal green sequence.
\end{cor}
\begin{proof}
	Let $\mathcal{C}'={^{\bot}\mu^{L}_{X}(\mathcal{N})}\cap \mathcal{N}^{\bot}$ and $\widetilde{\mathcal{C}'}=\mathcal{C}'/[\mathcal{N}_{X}]$. As in the previous corollary, we have isomorphisms of posets $\mathsf{silt}\,[\mu^{L}_{X}(\mathcal{N}),\mathcal{N}]\cong \mathsf{silt}\,\mathcal{C}'\cong \mathsf{silt}\,\widetilde{\mathcal{C}'}\leftrightarrow s\tau\text{-}\mathsf{tilt}\,A$. By \cite[Proposition 4.9]{BST19}, $A$ has a maximal green sequence if and only if there is a finite path of left mutations of support $\tau$-tilting pairs from $(A,0)$ to $(0,A)$. It is equivalent to a finite sequence of irreducible left mutations of silting objects in $\widetilde{\mathcal{C}'}$. Thus the result follows by Lemma \ref{mutation invariant}.
\end{proof}

\section{Relation with 2-Calabi-Yau reduction}\label{section 6}

In this section we study the relation between reduction in 0-Auslander categories and 2-Calabi-Yau reduction. To be more general, we consider the reduction technique introduced by E. Faber, B. R. Marsh and M. Pressland recently in \cite{FMP} (Theorem \ref{main thm in FMP}), which generalizes the classical reduction for 2-Calabi-Yau triangulated categories (see \cite{IY08}). Let $\mathcal{C}=(\mathcal{C},\mathbb{E},\mathfrak{s})$ be an extriangulated category.
A subcategory $\mathcal{T}$ of $\mathcal{C}$ is {\em rigid} if  $\mathbb{E}(\mathcal{T},\mathcal{T})=0$. It is {\em cluster-tilting} if ${^{\bot_{1}}\mathcal{T}}=\mathcal{T}=\mathcal{T}^{\bot_{1}}$ and it is functorially finite in $\mathcal{C}$ (see \cite{CZZ}, \cite{FMP}). Denote by $\mathsf{c\text{-}tilt}\,\mathcal{C}$ the set of all cluster-tilting subcategories of $\mathcal{C}$. For a subcategory $\mathcal{X}$,  denote by $\mathsf{c\text{-}tilt}_{\mathcal{X}}\mathcal{C}$ the set of all cluster-tilting subcategories of $\mathcal{C}$ containing $\mathcal{X}$.

\begin{defn}
	(1) \cite[Definition 7.1]{NP} $\mathcal{C}$ is called {\em Frobenius} if it has enough projectives $\mathcal{P}$ and enough injectives $\mathcal{I}$ and $\mathcal{P}=\mathcal{I}$.
	
	(2) (\cite[Definition 2.10]{CZZ}, \cite{FMP}) A Frobenius extriangulated category $\mathcal{C}$ is called {\em stably 2-Calabi-Yau} if it is $k$-linear ($k$ is a field) and the stable category $\mathcal{C}/[\mathcal{P}]$ is 2-Calabi-Yau triangulated category (it is triangulated, see \cite[Corollary 7.4]{NP}).
\end{defn}

Note that if $\mathcal{C}$ is stably 2-Calabi-Yau Frobenius, then for any subcategory $\mathcal{X}\subseteq \mathcal{C}$ we have ${^{\bot_{1}}\mathcal{X}}=\mathcal{X}^{\bot_{1}}$. For the rest of this section, we assume $\mathcal{C}$ is stably 2-Calabi-Yau Frobenius with projective-injectives $\mathcal{P}$ and $\mathcal{T}$ is a contravariantly finite subcategory of $\mathcal{C}$ such that $\mathcal{T}=\mathcal{T}^{\bot_{1}}$. Let $\mathcal{X}$ be a functorially finite rigid subcategory of $\mathcal{C}$. We recall the reduction theorem in \cite{FMP}.

\begin{thm}\cite[Theorem 3.21]{FMP}\label{main thm in FMP}
	The subcategory $\mathcal{X}^{\bot_{1}}$ in $\mathcal{C}$ is functorially finite and stably 2-Calabi-Yau Frobenius with projective-injectives $\mathsf{add}(\mathcal{X}\cup \mathcal{P})$. Moreover the cluster-tilting subcategories of $\mathcal{X}^{\bot_{1}}$ are precisely those cluster-tilting subcategories of $\mathcal{C}$ containing $\mathcal{X}$, that is,  $\mathsf{c\text{-}tilt}_{\mathcal{X}}\mathcal{C}\cong \mathsf{c\text{-}tilt}(\mathcal{X}^{\bot_{1}})$.
\end{thm}

\begin{prop}\label{relative struc on stably 2-CY}
	Let $\mathbb{E}_{\mathcal{T}}(C,A):=\{\delta\in \mathbb{E}(C,A)\,|\,\forall f:T\rightarrow C\text{ with }T\in \mathcal{T},f^{\ast}\delta=0\}$. Then the following holds.
	
	(1) $(\mathcal{C},\mathbb{E}_{\mathcal{T}})$ is  0-Auslander with projectives $\mathcal{T}$ and injectives $\mathcal{T}'=\mu^{L}(\mathcal{T};\mathcal{P})$ (mutation in $(\mathcal{C},\mathbb{E}_{\mathcal{T}})$).
	
	(2) For any $X,Y\in \mathcal{C}$, $\mathbb{E}(X,Y)=0=\mathbb{E}(Y,X)$ if and only if $\mathbb{E}_{\mathcal{T}}(X,Y)=0=\mathbb{E}_{\mathcal{T}}(Y,X)$.
	
	(3) A subcategory $\mathcal{N}$ is rigid (resp. cluster-tilting) in $(\mathcal{C},\mathbb{E})$ if and only if it is presilting (resp. functorially finite silting) in $(\mathcal{C},\mathbb{E}_{\mathcal{T}})$.
\end{prop}
\begin{proof}
	(1) Clearly we have $\mathcal{P}\subseteq \mathcal{T}$. By Example \ref{examples} (3), ${\rm Cone}_{(\mathcal{C},\mathbb{E},\mathfrak{s})}(\mathcal{T},\mathcal{T})$ is extension closed in $(\mathcal{C},\mathbb{E}_{\mathcal{T}})$ and is 0-Auslander. It suffices to show ${\rm Cone}_{(\mathcal{C},\mathbb{E},\mathfrak{s})}(\mathcal{T},\mathcal{T})=\mathcal{C}$.  
	Since $\mathcal{T}$ is contravariantly finite and contains $\mathcal{P}$, by \cite[Corollary 3.16]{NP}, for each $X\in \mathcal{C}$, there is an $\mathbb{E}$-triangle
	\begin{equation}\label{eq1}
		T^{1}\rightarrowtail T^{0}\stackrel{f}\twoheadrightarrow X\stackrel{\delta}\dashrightarrow
	\end{equation}
	with $f$ a right $\mathcal{T}$-approximation. Applying ${\rm Hom}_{(\mathcal{C},\mathbb{E})}(\mathcal{T},-)$ to \eqref{eq1}, we obtain $T^{1}\in \mathcal{T}$. By Example \ref{examples} (3), $\mathcal{T}$ and $\mathcal{T}'$ are projectives and injectives in $(\mathcal{C},\mathbb{E}_{\mathcal{T}})$.
	
	(2) Assume $\mathbb{E}_{\mathcal{T}}(X,Y)=0=\mathbb{E}_{\mathcal{T}}(Y,X)$. Applying the functor ${\rm Hom}_{(\mathcal{C},\mathbb{E})}(-,Y)$ to \eqref{eq1}, we obtain an exact sequence ${\rm Hom}_{\mathcal{C}}(T^{1},Y)\stackrel{\delta^{\sharp}}\rightarrow \mathbb{E}(X,Y)\stackrel{f^{\ast}}\rightarrow \mathbb{E}(T^{0},Y)$. Since $\delta \in \mathbb{E}_{\mathcal{T}}(X,T^{1})$, the image of $\delta^{\sharp}$ lies in $\mathbb{E}_{\mathcal{T}}(X,Y)=0$. Thus $f^{\ast}$ is injective. Let $\mathbb{D}:={\rm Hom}_{k}(-,k)$. Consider the following commutative diagram
	\[\begin{tikzcd}
		\mathbb{D}\mathbb{E}(T^{0},Y) \arrow[r,"\mathbb{D}f^{\ast}"] \arrow[d,"\simeq"] & \mathbb{D}\mathbb{E}(X,Y) \arrow[d,"\simeq"] \\
		\mathbb{E}(Y,T^{0}) \arrow[r,"f_{\ast}"] & \mathbb{E}(Y,X).
	\end{tikzcd}\]
	Since $T^{0}\in \mathcal{T}$, we have $\mathbb{E}(Y,T^{0})=\mathbb{E}_{\mathcal{T}}(Y,T^{0})$ by definition. The image of the surjection $f_{\ast}$ lies in $\mathbb{E}_{\mathcal{T}}(Y,X)=0$. Thus $\mathbb{E}(Y,X)=0$. Similarly we have $\mathbb{E}(X,Y)=0$. The other direction is obvious.
	
	(3) By (2), $\mathcal{N}$ is rigid in $(\mathcal{C},\mathbb{E})$ if and only if it is rigid in $(\mathcal{C},\mathbb{E}_{\mathcal{T}})$, which is equivalent to being presilting by (1). If $\mathcal{N}$ is cluster-tilting in $(\mathcal{C},\mathbb{E})$, then $\mathcal{N}={^{\bot_{1}}\mathcal{N}}\cap \mathcal{N}^{\bot_{1}}$ in $(\mathcal{C},\mathbb{E})$. By (2), it equals to ${^{\bot_{1}}\mathcal{N}}\cap \mathcal{N}^{\bot_{1}}$ in $(\mathcal{C},\mathbb{E}_{\mathcal{T}})$. Since $\mathcal{N}$ is functorially finite, by \cite[Theorem 4.3]{GNP23}, $\mathcal{N}$ is silting. If $\mathcal{N}$ is functorially finite silting in $(\mathcal{C},\mathbb{E}_{\mathcal{T}})$, then $\mathcal{N}={^{\bot_{1}}\mathcal{N}}\cap \mathcal{N}^{\bot_{1}}$ in $(\mathcal{C},\mathbb{E}_{\mathcal{T}})$ by Theorem \ref{AT22thm5.7}. Since $(\mathcal{C},\mathbb{E})$ is stably 2-Calabi-Yau, $\mathcal{N}={^{\bot_{1}}\mathcal{N}}=\mathcal{N}^{\bot_{1}}$ in $(\mathcal{C},\mathbb{E})$ by (2). Thus $\mathcal{N}$ is cluster-tilting in $(\mathcal{C},\mathbb{E})$.
\end{proof}

In the 0-Auslander category $(\mathcal{C},\mathbb{E}_{\mathcal{T}})$, $\mathcal{X}$ is presilting by Proposition \ref{relative struc on stably 2-CY} (3). As in Lemma \ref{two completions}, we define the Bongartz and co-Bongartz completions of $\mathcal{X}$ in $(\mathcal{C},\mathbb{E}_{\mathcal{T}})$. Indeed, for each $T\in \mathcal{T}$, choose an $\mathbb{E}_{\mathcal{T}}$-triangle $T\stackrel{f}\rightarrowtail X\twoheadrightarrow M_{T}\dashrightarrow $ with $f$ a left $\mathsf{add}(\mathcal{X}\cup \mathcal{P})$-approximation. Define $\mathcal{M}:=\mathsf{add}(\mathcal{X}\cup \mathcal{P}\cup \{M_{T}\,|\,T\in \mathcal{T}\})$. Then $\mathcal{M}$ is silting in $(\mathcal{C},\mathbb{E}_{\mathcal{T}})$ (the co-Bongartz completion of $\mathcal{X}$) and it is independent of the choice of $f$ as in Definition \ref{silt mutation in extri}. Dually for each $T'\in \mathcal{T}'$, choose an $\mathbb{E}_{\mathcal{T}}$-triangle $N_{T'}\rightarrowtail X'\stackrel{g}\twoheadrightarrow T'\dashrightarrow $ with $g$ a right $\mathsf{add}(\mathcal{X}\cup \mathcal{P})$-approximation. Define $\mathcal{N}:=\mathsf{add}(\mathcal{X}\cup \mathcal{P}\cup \{N_{T'}\,|\,T'\in \mathcal{T}'\})$. Then $\mathcal{N}$ is silting in $(\mathcal{C},\mathbb{E}_{\mathcal{T}})$ (the Bongartz completion of $\mathcal{X}$) and it is independent of the choice of $g$.

Similar to Lemma \ref{interval=summand completion}, \ref{C'is 0-Aus} and Remark \ref{rem on mutaion vs 0-Aus}, we have the following results. For convenience, we provide proofs.

\begin{lem}\label{silt contain X}
	In the 0-Auslander category $(\mathcal{C},\mathbb{E}_{\mathcal{T}})$, the following hold.
	
	(1) $\mathcal{M}\leq \mathcal{N}$ and $\mathsf{silt}\,[\mathcal{M},\mathcal{N}]=\mathsf{silt}_{\mathsf{add}(\mathcal{X}\cup \mathcal{P})}(\mathcal{C},\mathbb{E}_{\mathcal{T}})=\mathsf{silt}_{\mathcal{X}}(\mathcal{C},\mathbb{E}_{\mathcal{T}})$.
	
	(2) Let $\mathcal{C}':={^{\bot}\mathcal{M}}\cap \mathcal{N}^{\bot}$. Then it is 0-Auslander and $\mathcal{M}\cap \mathcal{N}=\mathsf{add}(\mathcal{X}\cup \mathcal{P})$.
	
	(3) ${^{\bot_{1}}\mathcal{X}}\cap \mathcal{X}^{\bot_{1}}=\mathcal{C}'$.
\end{lem}
\begin{proof}
	(1) Since $\mathcal{P}$ is projective-injective in $(\mathcal{C},\mathbb{E}_{\mathcal{T}})$, the second equality is obvious. For silting $\mathcal{T}$ such that $\mathcal{M}\leq \mathcal{T}\leq \mathcal{N}$, we have $\mathcal{X}\subseteq {^{\bot}\mathcal{N}}\cap \mathcal{M}^{\bot}\subseteq {^{\bot}\mathcal{T}}\cap \mathcal{T}^{\bot}=\mathcal{T}$. Let $\mathcal{T}$ be a silting subcategory that contains $\mathcal{X}$. Applying ${\rm Hom}_{(\mathcal{C},\mathbb{E}_{\mathcal{T}})}(\mathcal{T},-)$ to the $\mathbb{E}_{\mathcal{T}}$-triangle $T\stackrel{f}\rightarrowtail X\twoheadrightarrow M_{T}\dashrightarrow $, we obtain $\mathcal{T}\geq \mathcal{M}$. Dually we have $\mathcal{T}\leq \mathcal{N}$.
	
	(2) Let $T\rightarrowtail P\twoheadrightarrow T'\dashrightarrow$ be an $\mathbb{E}_{\mathcal{T}}$-triangle with $T\in \mathcal{T},\, T'\in \mathcal{T}'$ and $P\in \mathcal{P}$. Consider the following commutative diagram
	\[\begin{tikzcd}
		T \arrow[r,tail,"f"] \arrow[d,tail] & X \arrow[r,two heads] \arrow[d,tail] & M_{T} \arrow[r,dashed] \arrow[d,equal] & {} \\
		P \arrow[r,tail] \arrow[d,two heads] & P\oplus M_{T} \arrow[r,two heads] \arrow[d,two heads] & M_{T} \arrow[r,dashed] & {} \\
		T' \arrow[r,equal] \arrow[d,dashed] & T' \arrow[d,dashed] &  & \\
		{} & {} &  &
	\end{tikzcd}\]
	Then we obtain the commutative diagram
	\[\begin{tikzcd}
		& X \arrow[r,equal] \arrow[d,tail] & X \arrow[d,tail] & \\
		N_{T'} \arrow[r,tail] \arrow[d,equal] & X\oplus X' \arrow[r,two heads] \arrow[d,two heads] & P\oplus M_{T} \arrow[r,dashed] \arrow[d,two heads] & {} \\
		N_{T'} \arrow[r,tail] & X' \arrow[r,two heads,"g"] \arrow[d,dashed] & T' \arrow[r,dashed] \arrow[d,dashed] & {} \\
		& {} & {} & 
	\end{tikzcd}\]
	The second row implies that $\mathsf{add}(\mathcal{X}\cup \mathcal{P})$ is good covariantly finite in $\mathcal{N}$ (thus $\mu^{L}(\mathcal{N};\mathsf{add}(\mathcal{X}\cup \mathcal{P}))$ exists) and $\mathcal{M}\subseteq \mu^{L}(\mathcal{N};\mathsf{add}(\mathcal{X}\cup \mathcal{P}))$. By \cite[Lemma 5.3]{AT22}, $\mathcal{M}=\mu^{L}(\mathcal{N};\mathsf{add}(\mathcal{X}\cup \mathcal{P}))$. The results follow from Theorem \ref{main thm: new 0-Aus} and Lemma \ref{basic lem on mutation}.
	
	(3) Clearly $\mathcal{C}'\subseteq {^{\bot_{1}}\mathcal{X}}\cap \mathcal{X}^{\bot_{1}}$, it suffices to show the other inclusion. Applying ${\rm Hom}_{(\mathcal{C},\mathbb{E}_{\mathcal{T}})}({^{\bot_{1}}\mathcal{X}}\cap \mathcal{X}^{\bot_{1}},-)$ to the $\mathbb{E}_{\mathcal{T}}$-triangle $N_{T'}\rightarrowtail X\oplus X'\twoheadrightarrow P\oplus M_{T}\dashrightarrow $, we obtain ${^{\bot_{1}}\mathcal{X}}\cap \mathcal{X}^{\bot_{1}}\subseteq {^{\bot}\mathcal{M}}$. Dually we have ${^{\bot_{1}}\mathcal{X}}\cap \mathcal{X}^{\bot_{1}}\subseteq \mathcal{N}^{\bot}$. The result follows.
\end{proof}

\begin{cor}\label{reduction in 0-Aus_2}
	The map $\mathcal{S}\mapsto \mathcal{S}$ induces isomorphism of posets
	\[\mathsf{silt}_{\mathcal{X}}(\mathcal{C},\mathbb{E}_{\mathcal{T}})\cong \mathsf{silt}(\mathcal{C}',\mathbb{E}_{\mathcal{T}}|_{\mathcal{C}'}).\]
\end{cor}
\begin{proof}
	It follows from Lemma \ref{silt contain X} and Theorem \ref{main thm_1}.
\end{proof}

\begin{rem}
	Lemma \ref{silt contain X} and Corollary \ref{reduction in 0-Aus_2} hold in any 0-Auslander category $\mathcal{C}$ with functorially finite rigid subcategory $\mathcal{X}$. They are subcategory version of those results in Section \ref{Reduction} and \ref{Reduction implies mutation}. We refer to all these results as reduction theory in 0-Auslander categories.
\end{rem}

Since $\mathcal{X}^{\bot_{1}}={^{\bot_{1}}\mathcal{X}}\cap \mathcal{X}^{\bot_{1}}$ in $(\mathcal{C},\mathbb{E})$, which equals to ${^{\bot_{1}}\mathcal{X}}\cap \mathcal{X}^{\bot_{1}}=\mathcal{C}'$ in $(\mathcal{C},\mathbb{E}_{\mathcal{T}})$ by Proposition \ref{relative struc on stably 2-CY}, we study the relation between the two extriangulated structures on $\mathcal{X}^{\bot_{1}}$, namely $(\mathcal{X}^{\bot_{1}},\mathbb{E}|_{\mathcal{X}^{\bot_{1}}})$ and $(\mathcal{C}',\mathbb{E}_{\mathcal{T}}|_{\mathcal{C}'})$.

\begin{lem}\label{two 0-Aus struc coincide}
	The subcategory $\mathcal{N}$ is contravariantly finite and satisfies $\mathcal{N}=\mathcal{N}^{\bot_{1}}$ in $(\mathcal{X}^{\bot_{1}},\mathbb{E}|_{\mathcal{X}^{\bot_{1}}})$. Moreover $(\mathbb{E}|_{\mathcal{X}^{\bot_{1}}})_{\mathcal{N}}=\mathbb{E}_{\mathcal{T}}|_{\mathcal{C}'}$.
\end{lem}
\begin{proof}
	Since $(\mathcal{C}',\mathbb{E}_{\mathcal{T}}|_{\mathcal{C}'})$ is 0-Auslander, $\mathcal{N}$ is contravariantly finite in $(\mathcal{X}^{\bot_{1}},\mathbb{E}|_{\mathcal{X}^{\bot_{1}}})$. By the proof of Proposition \ref{relative struc on stably 2-CY} (3), we have $\mathcal{N}={^{\bot_{1}}\mathcal{N}}=\mathcal{N}^{\bot_{1}}$ in $(\mathcal{C},\mathbb{E})$. Thus $\mathcal{N}=\mathcal{N}^{\bot_{1}}$ in $(\mathcal{X}^{\bot_{1}},\mathbb{E}|_{\mathcal{X}^{\bot_{1}}})$. By Theorem \ref{main thm in FMP} and Proposition \ref{relative struc on stably 2-CY} (1), $(\mathcal{X}^{\bot_{1}},(\mathbb{E}|_{\mathcal{X}^{\bot_{1}}})_{\mathcal{N}})$ is 0-Auslander. We claim that it is just $(\mathcal{C}',\mathbb{E}_{\mathcal{T}}|_{\mathcal{C}'})$. For any $\mathbb{E}$-triangle $A\rightarrowtail B\stackrel{b}\twoheadrightarrow C\stackrel{\delta}\dashrightarrow $ with $A,C\in \mathcal{X}^{\bot_{1}}$. Assume $\delta\in \mathbb{E}_{\mathcal{T}}(C,A)$. Then for each $N_{T'}\in \mathcal{N}$ which is given by an $\mathbb{E}_{\mathcal{T}}$-triangle $N_{T'}\rightarrowtail X'\stackrel{g}\twoheadrightarrow T'\dashrightarrow $ as before, consider the following commutative diagram
	\[\begin{tikzcd}
		& T \arrow[r,equal] \arrow[d,tail,"a"] & T \arrow[d,tail] & \\
		N_{T'} \arrow[r,tail] \arrow[d,equal] & N_{T'}\oplus P \arrow[r,two heads] \arrow[d,two heads,"c"] & P \arrow[r,dashed] \arrow[d,two heads] & {} \\
		N_{T'} \arrow[r,tail] & X' \arrow[r,two heads,"g"] \arrow[d,dashed] & T' \arrow[r,dashed] \arrow[d,dashed] & {} \\
		& {} & {} &
	\end{tikzcd}\]
	where the third column is an $\mathbb{E}_{\mathcal{T}}$-triangle with $T\in \mathcal{T}$ and $P\in \mathcal{P}$. For any morphism $h:N_{T'}\oplus P\rightarrow C$, consider the following diagram
	\[\begin{tikzcd}
		& T \arrow[r,tail,"a"] \arrow[d,dashed,"d",swap] & N_{T'}\oplus P \arrow[r,two heads,"c"] \arrow[d,"h",near start,swap] \arrow[dl,dashed,"r",swap] & X' \arrow[r,dashed] \arrow[dl,dashed,"s"] \arrow[dll,dashed,"t",swap,near start] & {} \\
		A \arrow[r,tail] & B \arrow[r,two heads,"b",swap] & C \arrow[r,dashed,"\delta",swap] & {} &.
	\end{tikzcd}\]
	Since $\delta\in \mathbb{E}_{\mathcal{T}}(C,A)$, there exists $d$ such that $ha=bd$. Since $B\in \mathcal{X}^{\bot_{1}}$ and $X'\in \mathsf{add}(\mathcal{X}\cup \mathcal{P})$, there exists $r$ such that $ra=d$. Then $(h-br)a=0$, hence there exists $s$ such that $h=br+sc$. Since $A\in \mathcal{X}^{\bot_{1}}$, there exists $t$ such that $bt=s$. Thus $b(r+tc)=h$, which implies $\delta\in (\mathbb{E}|_{\mathcal{X}^{\bot_{1}}})_{\mathcal{N}}(C,A)$. The other inclusion is similar.
\end{proof}

Next, we recover the reduction in Theorem \ref{main thm in FMP} with Corollary \ref{reduction in 0-Aus_2}.

\begin{thm}\label{our reduction vs 2-CY reduction}
	There is a commutative diagram
	\[\begin{tikzcd}
		\mathsf{f\text{-}silt}_{\mathcal{X}}(\mathcal{C},\mathbb{E}_{\mathcal{T}}) \arrow[r,"\cong"] \arrow[d,equal] & \mathsf{f\text{-}silt}(\mathcal{C}',\mathbb{E}_{\mathcal{T}}|_{\mathcal{C}'}) \arrow[d,equal] \\
		\mathsf{c\text{-}tilt}_{\mathcal{X}}(\mathcal{C},\mathbb{E}) \arrow[r,"\cong"] & \mathsf{c\text{-}tilt}(\mathcal{X}^{\bot_{1}},\mathbb{E}|_{\mathcal{X}^{\bot_{1}}})
	\end{tikzcd}\]
	where $\mathsf{f\text{-}silt}(-)$ denote the functorially finite silting subcategories.
\end{thm}
\begin{proof}
	The top bijection follows from Corollary \ref{reduction in 0-Aus_2}, Theorem \ref{main thm in FMP}. The bottom one is given in Theorem \ref{main thm in FMP}. The vertical equalities follow from Proposition \ref{relative struc on stably 2-CY} (3) and Lemma \ref{two 0-Aus struc coincide}. The commutativity is obvious.
\end{proof}

\begin{rem}
	Compared with Theorem \ref{main thm in FMP}, we assume moreover the existence of $\mathcal{T}$ in our setting. It is natural since if  $\mathcal{C}$ has no contravariantly finite subcategories  $\mathcal{T}$ such that $\mathcal{T}=\mathcal{T}^{\bot_{1}}$, 
	then $\mathsf{c\text{-}tilt}_{\mathcal{X}}\mathcal{C}\cong \mathsf{c\text{-}tilt}(\mathcal{X}^{\bot_{1}})=\emptyset$.
\end{rem}

\begin{lem}\label{reduced case}
	Let $\overline{(-)}:\mathcal{C}\rightarrow \overline{\mathcal{C}}:=\mathcal{C}/[\mathcal{P}]$, then the following holds.
	
	(1) The functor $\overline{(-)}$ induces $\mathsf{c\text{-}tilt}(\mathcal{C},\mathbb{E})\cong \mathsf{c\text{-}tilt}(\overline{\mathcal{C}},\overline{\mathbb{E}})$.
	
	(2) $\mathcal{T}$ is contravariantly finite and satisfies $\mathcal{T}=\mathcal{T}^{\bot_{1}}$ in $(\overline{\mathcal{C}},\overline{\mathbb{E}})$. Moreover $\overline{\mathbb{E}}_{\mathcal{T}}=\overline{\mathbb{E}_{\mathcal{T}}}$, that is, $(\overline{\mathcal{C}},\overline{\mathbb{E}}_{\mathcal{T}})=(\overline{\mathcal{C}},\overline{\mathbb{E}_{\mathcal{T}}})$ as reduced 0-Auslander extriangulated categories.
\end{lem}
\begin{proof}
	(1) Since $\overline{\mathbb{E}}(C,A)=\mathbb{E}(C,A)$, for a subcategory $\mathcal{X}$, $\mathcal{X}^{\bot_{1}}$ in $(\mathcal{C},\mathbb{E})$ coincides with $\mathcal{X}^{\bot_{1}}$ in $(\overline{\mathcal{C}},\overline{\mathbb{E}})$. It suffices to show if $\mathcal{U}\in \mathsf{c\text{-}tilt}(\overline{\mathcal{C}},\overline{\mathbb{E}})$, then it is functorially finite in $(\mathcal{C},\mathbb{E})$. For each $Y$, assume $\overline{a}:U\rightarrow Y$ is a right $\mathcal{U}$-approximation in $(\overline{\mathcal{C}},\overline{\mathbb{E}})$ and $Y'\rightarrowtail P\stackrel{b}\twoheadrightarrow Y\dashrightarrow $ is an $\mathbb{E}$-triangle with $P\in \mathcal{P}$. Then $(a,b):U\oplus P\rightarrow Y$ is a right $\mathcal{U}$-approximation in $(\mathcal{C},\mathbb{E})$. Dually for left approximations.
	
	(2) It suffices to show the second part. Since $\overline{\mathbb{E}}(\overline{c},\overline{a})=\mathbb{E}(c,a)$ for any morphisms $a,c$ in $\mathcal{C}$ by \cite[Proposition 3.30]{NP}, it follows immediately by definition.
\end{proof}

\begin{prop}\label{commutative diagram reduced}
	There is a commutative diagram
	\[\begin{tikzcd}
		\mathsf{f\text{-}silt}(\mathcal{C},\mathbb{E}_{\mathcal{T}}) \arrow[r,"\cong"] \arrow[d,equal] & \mathsf{f\text{-}silt}(\overline{\mathcal{C}},\overline{\mathbb{E}_{\mathcal{T}}}) \arrow[d,equal] \\
		\mathsf{c\text{-}tilt}(\mathcal{C},\mathbb{E}) \arrow[r,"\cong"] & \mathsf{c\text{-}tilt}(\overline{\mathcal{C}},\overline{\mathbb{E}}).
	\end{tikzcd}\]
\end{prop}
\begin{proof}
	The top bijection follows from Lemma \ref{silting bijection} and the same argument as Lemma \ref{reduced case} (1). The bottom one is given in Lemma \ref{reduced case} (1). The vertical equalities follow from Proposition \ref{relative struc on stably 2-CY} (3) and Lemma \ref{reduced case} (2). The commutativity is obvious.
\end{proof}

\begin{rem}
	By Theorem \ref{our reduction vs 2-CY reduction} and Proposition \ref{commutative diagram reduced}, reduction theory in 0-Auslander categories recovers the classical cluster-tilting reduction for 2-Calabi-Yau triangulated categories. Hence Theorem \ref{main thm2} also covers \cite[Theorem 4.24]{G.Jasso}.
\end{rem}

\section{$d$-Auslander extriangulated categories}\label{section 7}

Let $\mathcal{C}=(\mathcal{C},\mathbb{E},\mathfrak{s})$ be an extriangulated category. In this section, we generalize Theorem \ref{main thm: new 0-Aus} to $d$-Auslander extriangulated categories as follows.

\begin{thm}\label{d-Auslander vs mutation}
Let $\mathcal{M},\mathcal{N}\in \mathsf{silt}\,\mathcal{C}$, $\mathcal{C}':={^{\bot}\mathcal{M}}\cap \mathcal{N}^{\bot}$ and $d\geq 0$ an integer. Then $\mathcal{M}\leq \mathcal{N}$ and $\mathcal{C}'$ is $d$-Auslander if and only if there is a sequence of silting subcategories 
\[\mathcal{N}=\mathcal{T}_{0}\geq \mathcal{T}_{1}\geq \cdots \geq \mathcal{T}_{d+1}=\mathcal{M}\]
such that for $0\leq i\leq d$, $\mathcal{T}_{i+1}=\mu^{L}(\mathcal{T}_{i};\mathcal{D})$ for a good covariantly finite subcategory $\mathcal{D}$ of $\mathcal{T}_{i}$.
\end{thm}

First we need to explore the definition of $d$-Auslander extriangulated categories. Recently, Xiaofa Chen \cite{Chen Xiaofa} also defined $d$-Auslander exact dg categories.

\begin{defn}\label{defn of d-Aus} \cite{GNP23}
	$\mathcal{C}$ is called {\em $d$-Auslander} ($d\geq 0$) if it has enough projectives and ${\rm pd}~\mathcal{C}\leq d+1\leq {\rm dom.dim}\,\mathcal{C}$. 
\end{defn}

\begin{prop}\label{equi defn for d-Aus}
	The following are equivalent.
	
	(1) $\mathcal{C}$ has enough projectives and ${\rm pd}~\mathcal{C}\leq d+1\leq {\rm dom.dim}\,\mathcal{C}$.
	
	(2) $\mathcal{C}$ has enough injectives and ${\rm id}~\mathcal{C}\leq d+1\leq {\rm codom.dim}\,\mathcal{C}$.
\end{prop}
\begin{proof}
	The case $d=0$ is proved in \cite[Proposition 3.6]{GNP23}. Thus we assume $d\geq 1$. We only prove (1) implies (2). The other implication is dual. Denote by $\mathcal{P}$ (resp. $\mathcal{I}$) the subcategory of all projectives (resp. injectives). Let $\mathcal{D}=\mathcal{P}\cap \mathcal{I}$. Then $\mathcal{P}$ is a silting subcategory. By definition $\mathcal{T}_{1}:=\mu^{L}(\mathcal{P};\mathcal{D})$ exists, hence it is silting. We claim that $\mathcal{D}$ is a good covariantly finite subcategory of $\mathcal{T}_{1}$ and for each $M_{1}\in \mathcal{T}_{1}$, there are $\mathbb{E}$-triangles
	\[M_{1}\rightarrowtail I_{1}\twoheadrightarrow M_{2}\dashrightarrow \]
	\[\cdots\]
	\[M_{d}\rightarrowtail I_{d}\twoheadrightarrow I_{d+1}\dashrightarrow \]
	with $I_{k}\in \mathcal{D}$ for $1\leq k\leq d$ and $I_{d+1}\in \mathcal{I}$. Since $\mathcal{P}$ and $\mathcal{T}_{1}$ are mutations of each other, there is an $\mathbb{E}$-triangle $P\rightarrowtail I_{0}\twoheadrightarrow M_{1}\dashrightarrow $ with $P\in \mathcal{P}$ and $I_{0}\in \mathcal{D}$. By definition there are $\mathbb{E}$-triangles
	\[P\rightarrowtail I_{0}'\twoheadrightarrow M_{1}'\dashrightarrow \]
	\[M_{1}'\rightarrowtail I_{1}'\twoheadrightarrow M_{2}'\dashrightarrow \]
	\[\cdots\]
	\[M_{d}'\rightarrowtail I_{d}'\twoheadrightarrow I_{d+1}'\dashrightarrow \]
	with $I_{k}'\in \mathcal{D}$ for $0\leq k\leq d$ and $I_{d+1}'\in \mathcal{I}$. Since $M_{1}\oplus I_{0}'\cong M_{1}'\oplus I_{0}$, consider the following commutative diagram by (ET4)
	\[\begin{tikzcd}
		M_{1} \arrow[r,equal] \arrow[d,tail] & M_{1} \arrow[d,tail] & & \\
		M_{1}'\oplus I_{0} \arrow[r,tail] \arrow[d,two heads] & I_{1}'\oplus I_{0} \arrow[r,two heads] \arrow[d,two heads] & M_{2}' \arrow[d,equal] \arrow[r,dashed] & {} \\
		I_{0}' \arrow[r,tail] \arrow[d,dashed] & I_{0}'\oplus M_{2}' \arrow[r,two heads] \arrow[d,dashed] & M_{2}' \arrow[r,dashed] & {} \\
		{} & {} & &
	\end{tikzcd}.\]
    Then we obtain $d$ $\mathbb{E}$-triangles
    \[M_{1}\rightarrowtail I_{1}'\oplus I_{0}\twoheadrightarrow I_{0}'\oplus M_{2}'\dashrightarrow\]
    \[M_{2}'\oplus I_{0}'\rightarrowtail I_{2}'\oplus I_{0}'\twoheadrightarrow M_{3}'\dashrightarrow \]
    \[M_{3}'\rightarrowtail I_{3}'\twoheadrightarrow M_{4}'\dashrightarrow \]
    \[\cdots\]
    \[M_{d}'\rightarrowtail I_{d}'\twoheadrightarrow I_{d+1}'\dashrightarrow .\]
    Thus the claim follows. Inductively we obtain a finite sequence of silting subcategories $\mathcal{P}=\mathcal{T}_{0}\geq \mathcal{T}_{1}\geq \cdots \geq \mathcal{T}_{d}\geq \mathcal{T}_{d+1}\subseteq \mathcal{I}$ such that $\mathcal{T}_{i+1}=\mu^{L}(\mathcal{T}_{i};\mathcal{D})$ for $0\leq i\leq d$. Then $\mathcal{T}_{d+1}=\mathcal{I}$ by \cite[Lemma 5.3]{AT22}. Thus $\mathcal{I}$ is silting, hence $\mathcal{C}$ has enough injectives and ${\rm id}\,\mathcal{C}\leq d+1$ by the dual of \cite[Proposition 5.5]{AT22}. Since $\mathcal{T}_{i+1}$ and $\mathcal{T}_{i}$ are mutations of each other, ${\rm codom.dim}\,\mathcal{C}\geq d+1$ follows immediately.
\end{proof}

\begin{proof}[Proof of Theorem \ref{d-Auslander vs mutation}]
	In Proposition \ref{equi defn for d-Aus}, we actually prove the necessity by Theorem \ref{main thm_1}. For sufficiency, by Corollary \ref{main cor} and Theorem \ref{AT22thm5.7}, it suffices to show $\mathbb{E}^{\geq d+2}(\mathcal{M}^{\vee},\mathcal{N}^{\wedge})=0$. For each $N\in \mathcal{N}$, there are $\mathbb{E}$-triangles
	\[N\rightarrowtail D_{0}\twoheadrightarrow T_{1}\dashrightarrow \]
	\[T_{1}\rightarrowtail D_{1}\twoheadrightarrow T_{2}\dashrightarrow \]
	\[\cdots\]
	\[T_{d}\rightarrowtail D_{d}\twoheadrightarrow M\dashrightarrow \]
	with $D_{k}\in \mathcal{D}$ for $0\leq k\leq d$ and $M\in \mathcal{M}$. Applying the functor ${\rm Hom}_{\mathcal{C}}(\mathcal{M}^{\vee},-)$, we obtain for $k\geq 0$,  $\mathbb{E}^{d+k+2}(\mathcal{M}^{\vee},N)\cong \mathbb{E}^{d+k+1}(\mathcal{M}^{\vee},T_{1})\cong \cdots \cong \mathbb{E}^{k+2}(\mathcal{M}^{\vee},T_{d})=0$ by the proof of Theorem \ref{main thm: new 0-Aus}. By the same argument as that of Theorem \ref{main thm: new 0-Aus}, we obtain $\mathbb{E}^{\geq d+2}(\mathcal{M}^{\vee},\mathcal{N}^{\wedge})=0$.
\end{proof}

\begin{prop}
	Let $\mathcal{T}_{i}\in \mathsf{silt}\,\mathcal{C},0\leq i\leq 2$ such that for $i=0,1$, $\mathcal{T}_{i+1}=\mu^{L}(\mathcal{T}_{i};\mathcal{D})$ for a good covariantly finite subcategory $\mathcal{D}$ of $\mathcal{T}_{0},\mathcal{T}_{1}$.
	
	(1) There are mutually inverse isomorphisms of posets
	\begin{align*}
		\mathsf{silt}\,[\mathcal{T}_{1},\mathcal{T}_{0}] & \leftrightarrow \mathsf{silt}\,[\mathcal{T}_{2},\mathcal{T}_{1}] \\
		\mathcal{N} & \mapsto \mu^{L}(\mathcal{N};\mathcal{D}) \\
		\mu^{R}(\mathcal{M};\mathcal{D}) & \mapsfrom \mathcal{M}.
	\end{align*}

    (2) Let $\mathcal{M},\mathcal{N}\in \mathsf{silt}\,[\mathcal{T}_{1},\mathcal{T}_{0}]$ and $\mathcal{D}'=\mathcal{M}\cap \mathcal{N}$, $\mathcal{D}''=\mu^{L}(\mathcal{M};\mathcal{D})\cap \mu^{L}(\mathcal{N};\mathcal{D})$. Then $\mathcal{M}=\mu^{L}(\mathcal{N};\mathcal{D}')$ if and only if $\mu^{L}(\mathcal{M};\mathcal{D})=\mu^{L}(\mu^{L}(\mathcal{N};\mathcal{D});\mathcal{D}'')$.
    
    (3) If $\mathcal{C}$ is $k$-linear ($k$ is a field), Hom-finite and Krull-Schmidt and $\mathcal{D}=(\mathcal{T}_{0})_{X}$ for an $X\in \mathcal{T}_{0}$. Let $A:=\frac{\rm End_{\mathcal{C}}(X)}{[\mathcal{D}](X,X)}$. There is a commutative diagram of bijections
    \[\begin{tikzcd}
    	\mathsf{silt}\,[\mathcal{T}_{1},\mathcal{T}_{0}] \arrow[rr,leftrightarrow,"(1)"] \arrow[dr] & & \mathsf{silt}\,[\mathcal{T}_{2},\mathcal{T}_{1}] \arrow[dl] \\
    	& s\tau\text{-}\mathsf{tilt}A &
    \end{tikzcd}\]
    where the left and right maps are given by Theorem \ref{main thm_1}, Lemma \ref{silting bijection} and Corollary \ref{silt-tau-tilt}.
\end{prop}
\begin{proof}
	(1) Let $\mathcal{C}':={^{\bot}\mathcal{T}_{1}}\cap \mathcal{T}_{0}^{\bot}$ and $\mathcal{N}\in \mathsf{silt}\,[\mathcal{T}_{1},\mathcal{T}_{0}]$. Then $\mathcal{C}'$ is 0-Auslander and $\mathcal{N}\in \mathsf{silt}\,\mathcal{C}'$ by Theorem \ref{main thm: new 0-Aus} and \ref{main thm_1} and $\mathcal{D}\subseteq \mathcal{N}$. For each $N\in \mathcal{N}$, there is a commutative diagram by (ET4)
	\[\begin{tikzcd}
		N \arrow[r,tail] \arrow[d,equal] & T_{1} \arrow[r,two heads] \arrow[d,tail] & T_{1}' \arrow[r,dashed] \arrow[d,tail] & {} \\
		N \arrow[r,tail,"f"] & D \arrow[r,two heads] \arrow[d,two heads] & M \arrow[r,dashed] \arrow[d,two heads] & {} \\
		& T_{2} \arrow[r,equal] \arrow[d,dashed] & T_{2} \arrow[d,dashed] & \\
		& {} & {} &
	\end{tikzcd}\]
    where the first row is the injective copresentation in $\mathcal{C}'$ (that is, $T_{1},T_{1}'\in \mathcal{T}_{1}$) and the second column is an $\mathbb{E}$-triangle with $D\in \mathcal{D}$ and $T_{2}\in \mathcal{T}_{2}$. Then $f$ is a left $\mathcal{D}$-approximation. Hence $\mu^{L}(\mathcal{N};\mathcal{D})$ exists. By the third column, we obtain $\mu^{L}(\mathcal{N};\mathcal{D})\in \mathsf{silt}\,[\mathcal{T}_{2},\mathcal{T}_{1}]$. Assume $\mathcal{N}_{1},\mathcal{N}_{2}\in \mathsf{silt}\,[\mathcal{T}_{1},\mathcal{T}_{0}]$ such that $\mathcal{N}_{1}\geq \mathcal{N}_{2}$. Then there are $\mathbb{E}$-triangles $N_{1}\stackrel{f_{1}}\rightarrowtail D_{1}\twoheadrightarrow M_{1}\dashrightarrow $ and $N_{2}\stackrel{f_{2}}\rightarrowtail D_{2}\twoheadrightarrow M_{2}\dashrightarrow $ where $N_{1}\in \mathcal{N}_{1},N_{2}\in \mathcal{N}_{2},D_{1},D_{2}\in \mathcal{D}$ and $f_{1},f_{2}$ are left $\mathcal{D}$-approximations. Applying ${\rm Hom}_{\mathcal{C}}(-,N_{2})$ to the first one, we obtain $\mathbb{E}^{\geq 2}(M_{1},N_{2})=0$. Applying ${\rm Hom}_{\mathcal{C}}(M_{1},-)$ to the second one, we obtain $\mathbb{E}^{\geq 1}(M_{1},M_{2})=0$. Hence $\mu^{L}(\mathcal{N}_{1};\mathcal{D})\geq \mu^{L}(\mathcal{N}_{2};\mathcal{D})$. The other direction is similar.
    
    (2) Clearly $\mathcal{D}\subseteq \mathcal{D}'$ and $\mathcal{D}''=\mu^{L}(\mathcal{D}';\mathcal{D})$. For each $X\in \mu^{L}(\mathcal{N};\mathcal{D})$, we have an $\mathbb{E}$-triangle $N\rightarrowtail D\twoheadrightarrow X\dashrightarrow $ with $N\in \mathcal{N},D\in \mathcal{D}$. There is a commutative diagram by (ET4)
    \[\begin{tikzcd}
    	N \arrow[r,tail] \arrow[d,equal] & D' \arrow[r,two heads] \arrow[d,tail] & M \arrow[d,tail] \arrow[r,dashed] & {} \\
    	N \arrow[r,tail] & D_{0} \arrow[r,two heads] \arrow[d,two heads] & X' \arrow[d,two heads] \arrow[r,dashed] & {} \\
    	& D'' \arrow[r,equal] \arrow[d,dashed] & D'' \arrow[d,dashed] & \\
    	& {} & {} &
    \end{tikzcd}\]
    where the first row and the second column are $\mathbb{E}$-triangles with $D'\in \mathcal{D}',M\in \mathcal{M},D_{0}\in \mathcal{D},D''\in \mathcal{D}''$. Then the second row is an $\mathbb{E}$-triangle with $X'\in \mu^{L}(\mathcal{N};\mathcal{D})$. Thus we have $X\oplus D_{0}\cong X'\oplus D$. Consider the following commutative diagram
    \[\begin{tikzcd}
    	M \arrow[r,tail] \arrow[d,tail] & X' \arrow[r,two heads] \arrow[d,tail] & D'' \arrow[r,dashed] \arrow[d,equal] & {} \\
    	D_{1} \arrow[r,tail] \arrow[d,two heads] & Z \arrow[r,two heads] \arrow[d,two heads] & D'' \arrow[r,dashed] & {} \\
    	Y \arrow[r,equal] \arrow[d,dashed] & Y \arrow[d,dashed] & & \\
    	{} & {} & &
    \end{tikzcd}\]
    where the first column is an $\mathbb{E}$-triangle with $D_{1}\in \mathcal{D}$ and $Y\in \mu^{L}(\mathcal{M};\mathcal{D})$. Then the second row splits and hence $Z\cong D_{1}\oplus D''\in \mathcal{D}''$. Consider the $\mathbb{E}$-triangle in the second column, as in the proof of Proposition \ref{equi defn for d-Aus}, we can obtain another $\mathbb{E}$-triangle $X\rightarrowtail Z\oplus D \twoheadrightarrow Y\oplus D_{0}\dashrightarrow$. Then $\mu^{L}(\mu^{L}(\mathcal{N};\mathcal{D}),\mathcal{D}'')$ exists and is a subcategory of $\mu^{L}(\mathcal{M};\mathcal{D})$. By \cite[Lemma 5.3]{AT22}, they are equal. Dually we obtain the other half. Thus the result follows.
        
    (3) Take an $\mathbb{E}$-triangle $X\stackrel{f}\rightarrowtail D\twoheadrightarrow Y\dashrightarrow$ with $D\in \mathcal{D}$ and $Y\in \mathcal{T}_{1}$. Let $A':=\frac{\rm End_{\mathcal{C}}(Y)}{[\mathcal{D}](Y,Y)}$. Since the subcategory ${^{\bot}\mathcal{T}_{1}}\cap \mathcal{T}_{0}^{\bot}$ is 0-Auslander by Theorem \ref{main thm: new 0-Aus}, we obtain $A\cong A'$ by Lemma \ref{properties of 0-Aus} (4). For $\mathcal{N}\in \mathsf{silt}\,[\mathcal{T}_{1},\mathcal{T}_{0}]$, let $N\in \mathcal{N}$ such that $\mathcal{N}=\mathsf{add}N$ in $({^{\bot}\mathcal{T}_{1}}\cap \mathcal{T}_{0}^{\bot})/[\mathcal{D}]$. Take an $\mathbb{E}$-triangle $N\stackrel{g}\rightarrowtail D_{0}\twoheadrightarrow M\dashrightarrow$ with $D_{0}\in \mathcal{D},M\in \mu^{L}(\mathcal{N};\mathcal{D})$. It suffices to show ${\rm Hom}_{\mathcal{C}/[\mathcal{D}]}(X,N)\cong {\rm Hom}_{\mathcal{C}/[\mathcal{D}]}(Y,M)$ as right $A$-modules. In the reduced 1-Auslander extriangulated category $\widetilde{\mathcal{C}'}:=({^{\bot}\mathcal{T}_{2}}\cap \mathcal{T}_{0}^{\bot})/[\mathcal{D}]$, consider a morphism of $\mathbb{E}_{\widetilde{\mathcal{C}'}}$-triangles
    \[\begin{tikzcd}
    	X \arrow[r,tail,"\widetilde{f}"] \arrow[d,"\widetilde{a}"] & D \arrow[r,two heads] \arrow[d] & Y \arrow[r,dashed] \arrow[d,"\widetilde{c}"] & {} \\
    	N \arrow[r,tail,"\widetilde{g}"] & D_{0} \arrow[r,two heads] & M \arrow[r,dashed] & {}.
    \end{tikzcd}\]
    For each $\widetilde{a}\in {\rm Hom}_{\mathcal{C}/[\mathcal{D}]}(X,N)$ there is a unique $\widetilde{c}\in {\rm Hom}_{\mathcal{C}/[\mathcal{D}]}(Y,M)$ such that $(\widetilde{a},0,\widetilde{c})$ is a morphism of $\mathbb{E}_{\widetilde{\mathcal{C}'}}$-triangles and vice versa. Thus ${\rm Hom}_{\mathcal{C}/[\mathcal{D}]}(X,N)\cong {\rm Hom}_{\mathcal{C}/[\mathcal{D}]}(Y,M)$ and it is a right $A$-module isomorphism.
\end{proof}

The following is a direct consequence.

\begin{cor}
	Assume $\mathcal{C}$ is a $k$-linear ($k$ is a field), Hom-finite, Krull-Schmidt and $d$-Auslander ($d\geq 0$) with projectives $\mathcal{P}$ and injectives $\mathcal{I}$. Let $\mathcal{D}:=\mathcal{P}\cap \mathcal{I}$ and $\widetilde{\mathcal{C}}:=\mathcal{C}/[\mathcal{D}]$. If $\mathcal{P}=\mathsf{add}P$ in $\widetilde{\mathcal{C}}$ for an object $P$. Let $A:={\rm End}_{\widetilde{\mathcal{C}}}(P)$. Then $\#\mathsf{silt}\,\mathcal{C}\geq (d+1)\#s\tau\text{-}\mathsf{tilt}A-d$, where $\#(-)$ denote the cardinality of a set.
\end{cor}

\subsection*{Acknowledgments} The authors would like to thank Yifei Cheng for useful discussions. The work is supported partially by National Natural Science Foundation of China Grant No. 12031007, 12371034.

\bibliographystyle{plain}

\end{document}